\documentclass[12pt]{article}

\usepackage{estilo}

\begin{document}

	\begin{center}
		\rule{15cm}{1.5pt} \vspace{.6cm}
		
		{\Large \bf Comparison methods for semilinear elliptic problems\\[3mm] on Riemannian manifolds with a Ricci lower bound}
		
		\vspace{0.4cm}
		
		\author[Jos\'{e} M. Espinar$\mbox{}^1$, Fern\'an Gonz\'alez-Ib\'a\~{n}ez$\mbox{}^1$ and Diego A. Mar\'{i}n$\mbox{}^1$
		
		\vspace{0.3cm} \rule{15cm}{1.5pt}
	\end{center}
	
	\vspace{.3cm}

	\noindent $\mbox{}^1$ Department of Geometry and Topology and Institute of Mathematics (IMAG), University of Granada, 18071, Granada, Spain;\vspace{0.2cm}
	e-mail:\{jespinar, fernangi, damarin\}@ ugr. es
	
	\vspace{.3cm}

\begin{abstract}
In the first part of the article we develop a comparison method for positive solutions of the semilinear Dirichlet problem $\Delta u+f(u)=0$ on domains $\Omega\subset \mathcal M^n$ of a Riemannian manifold $(\mathcal{M}^n,g)$ with a Ricci lower bound $\operatorname{Ric}_g\ge (n-1)k\,g$. 
Assuming admissibility and structural conditions on $f$, we prove a sharp pointwise gradient comparison, with a rigid characterization of the equality case. 
As applications, we derive an explicit isoperimetric-type inequality and a quantitative hot-spot localization estimate under natural convexity assumptions.
In the second part, on $\mathbb S^n$ we show that isoparametric foliations produce non-rotational $f$-extremal domains, and that these examples descend to smooth quotients under free isometric actions preserving the foliation.
\end{abstract}



\section{Introduction}\label{sect_Introduction}
In elliptic PDE one often encounters a tight interplay between the geometry of a domain $\Omega$ and the analytic features encoded by a solution $u$. 
Global scales such as the inradius and the volume, together with local boundary information (for instance, bounds on principal curvatures or on the mean curvature of $\partial\Omega$), typically constrain the maximum value of $u$, the shape and location of its maximum set, and the behaviour of $|\nabla u|$ both on $\partial\Omega$ and along level sets.
A paradigmatic manifestation of this rigidity is Serrin’s overdetermined torsion problem on $\mathbb{R}^n$ ($\Delta u=-1$ in $\Omega$, $u=0$ on $\partial\Omega$): if $\Omega$ is bounded, $\partial\Omega$ is sufficiently regular, and $|\nabla u|$ is constant on $\partial\Omega$, then $\Omega$ must be a ball; see \cite{Se,We}.
In a different direction, for convex annular domains one can obtain lower bounds for the smallest principal curvature of level sets of harmonic solutions and of solutions to certain semilinear equations, in terms of boundary geometry and the boundary size of $|\nabla u|$ \cite{alicechangPrincipalCurvatureEstimates2010}.
Also, the location of the maximum set can be quantified in terms of intrinsic scales of the domain: for a solution to the torsion problem under assumptions such as mean convexity of $\partial\Omega$, one derives estimates forcing maximum points to stay a controlled distance away from the boundary, with the inradius as the relevant scale \cite{MP}.

Building on these ideas, Agostiniani, Borghini, and Mazzieri \cite{ABM}, among other results, introduced a comparison method that associates to a torsion solution in a planar annular domain a suitably chosen radial model solution, selected from an explicit one-parameter family.
This comparison yields \textit{a priori} control of $|\nabla u|$ and, in turn, a qualitative description of the solution, including bounds for the location of the maximum set and curvature estimates for suitable level sets.

The same philosophy was subsequently developed in Euclidean space in arbitrary dimension in \cite{ABBM}, and extended to the spherical setting in dimension $2$ for more general problems in \cite{EMa,EMa2}, with further applications, including contributions to the study of the critical catenoid conjecture.
Related comparison principles also arise in the analysis of static solutions of the Einstein equations with positive cosmological constant; see \cite{Chr,borghiniMassStaticMetrics2018}.

In this paper we extend the comparison scheme of \cite{EMa,EMa2} to Riemannian manifolds $(\mathcal M^n,g)$ under a Ricci lower  curvature bound,
\begin{equation}\label{eq:Ric}
\operatorname{Ric}_g \ge (n-1)k\,g .
\end{equation}

This direction was proposed in \cite[Section 1.5]{ABM}. Our approach combines comparison geometry with a two--parameter family of one-dimensional model profiles obtained by imposing radial symmetry in a suitable model space. The analytic starting point is the semilinear Dirichlet problem for the Laplace--Beltrami operator.

Let $\Omega\subset \mathcal M$ be a domain with $\mathcal{C}^2$ boundary and let $f\in \mathrm{Lip}_{\mathrm{loc}}(\mathbb{R})$. We consider
\begin{equation}\label{DP}
\left\{
\begin{array}{rlll}
\Delta u + f(u) = 0 \quad &\text{in } \Omega,\\
u>0 \quad &\text{in } \Omega,\\
u=0 \quad &\text{on } \partial\Omega,
\end{array}
\right.
\end{equation}
where $\Delta$ denotes the Laplace--Beltrami operator with respect to $g$. This formulation includes, as special cases, Lane--Emden--Fowler equations, Allen--Cahn-type nonlinearities, and affine Helmholtz problems.

The comparison scheme hinges on the availability of radial model solutions (on disks or annuli) that are positive in the interior, vanish on the boundary, and are strictly monotone in the radial variable away from the maximum set; since these qualitative features are not automatic for a general locally Lipschitz $f$, we encode them through the notion of $k$--admissibility (Definition~\ref{def:admissible-f}). The models themselves are built in a boundary-adapted way: we introduce a warped product metric on an interval times a connected component of $\partial\Omega$ endowed with the induced metric, and we restrict to fiber-invariant functions so that \eqref{DP} reduces to an ODE. Varying the maximum value and the radial location where it is attained produces a two-parameter family of explicit profiles, which generate model solutions on canonical radial domains (of disk- and annulus-type). To compare a given pair $(\mathcal{U},u)$, where $\mathcal{U}\in \pi_0(\Omega\setminus \operatorname{Max}(u))$ is a connected component of the complement of the maximum set $\operatorname{Max}(u):=\{\,p\in\Omega:\ u(p)=\max_{\Omega}u\,\}$, we introduce a boundary parameter extracted from extremal values of $|\nabla u|^2$ on $\partial\Omega$, in the spirit of \cite{EMa2, ABM, ABBM,AFM}. This parameter selects an appropriate comparison pair (not necessarily uniquely, unlike in \cite{ABM, ABBM}); once such a pair is fixed, the argument proceeds as in \cite{ABM} and yields the desired estimates.

The first main result is a sharp pointwise gradient comparison (Theorem~\ref{theo:Gradient}).
Roughly speaking, each region between the maximum set of $u$ and the boundary can be compared with an explicit one-dimensional profile. Under the structural assumptions of Theorem~\ref{theo:Gradient}, this yields a pointwise bound for $|\nabla u|$ in terms of the corresponding model gradient.
Moreover, the equality case is rigid: if the estimate is sharp at one point, then the geometry of the region and the solution itself are isometric to the model.

A second consequence is an integral isoperimetric-type estimate (Theorem~\ref{theo:Isoperimetric}).
Assuming a mild regularity hypothesis on the $(n-1)$-dimensional part $\Gamma_M$ of $\mathrm{Max}(u)$, the pointwise gradient comparison can be combined with a level-set area bound and the coarea formula to obtain an explicit lower bound for the ratio $\mathcal H^n(\mathcal U)/\mathcal H^{n-1}(\Gamma_M)$ in terms of the corresponding model profile.
Moreover, equality forces $(\mathcal U,u)$ to be isometric to the comparison model pair.

The third main result provides a quantitative localization of hot spots in the spirit of \cite{MP} (Theorem~\ref{theo:HotSpots}), extending \cite[Theorem~5.2]{ABM} to the present semilinear setting.
Besides the assumptions of Theorem~\ref{theo:Gradient}, one imposes natural pinching  condition on $f$ and convexity-type hypotheses on $\Omega$.
Under these assumptions, we obtain an explicit lower bound for $\mathrm{dist}(\mathrm{Max}(u),\partial\Omega)$ in terms of the inradius and the corresponding model parameters.
The equality case is fully rigid: it forces the ambient manifold to be a space form of sectional curvature $k$, denoted as $\mathbb{M}^n(k)$, the nonlinearity to be affine, namely $f(x)=nkx+1$, and $(\Omega,u)$ to coincide, up to congruence, with the Serrin-type ball solution (recovering Serrin’s torsion equation when $k=0$).

The remainder of the paper is devoted to the study of overdetermined elliptic problems (OEP) on the sphere $(\mathbb{S}^n,g_{\mathbb{S}^n})$ and to the construction of $f$-extremal domains, in the sense of \cite{Ros}. Model solutions exhibit a radial structure, since they are obtained from the ordinary differential equation associated with the problem. As a consequence, the norm of the gradient is constant along regular level sets and, in particular, the normal derivative is constant on each boundary component of the model domain. Therefore, these solutions provide natural examples of \(f\)-extremal domains.

The sphere is, moreover, a fundamental example of a model manifold, and in this setting the picture is especially rich, see for example \cite{ruizOverdeterminedEllipticProblems2023a,FMW}. The existence of many isoparametric foliations in  spheres provides a natural geometric mechanism for constructing \(f\)-extremal domains beyond the radial setting. In particular, this approach yields non-rotationally symmetric examples, such as those constructed by Shklover \cite{Shk} and Savo \cite{Savo}, and can be extended to nonlinearities satisfying suitable conditions analogous to \(k\)--admissibility; see Definition \ref{def:admissibleIsoparametric-f}. Moreover, we prove that, for every closed subgroup of the isometry group of the sphere that preserves the foliation and acts freely on the relevant leaves, these examples descend to smooth quotients.

\paragraph{Organisation of the paper.} 
In Section~\ref{sec:ComMethRes} we construct the comparison framework and state the main results.
Section~\ref{sectionGradient} contains the proofs of these comparison results.
Section~\ref{sec:ExoticSol} discusses the construction of non-rotationally symmetric solutions on the sphere and their descent to quotients.
Appendix~\ref{appendixSolutions} collects the ODE statements used in Sections~\ref{sec:ComMethRes} and~\ref{sec:ExoticSol}, and Appendix~\ref{appendixGradient} proves an elliptic inequality needed for the gradient estimates.

\section{Comparison method and its results}\label{sec:ComMethRes}

In this section, we set up the comparison framework, extending the strategy of \cite{EMa,EMa2} to solutions of \eqref{DP} under the Ricci lower bound \eqref{eq:Ric}. 
We work with boundary-adapted model manifolds realized as warped products and with the associated model solutions by reducing \eqref{DP} to a radial ODE for fiber-invariant functions and then lifting the resulting profiles to canonical radial domains (of disk- and annulus-type). 
A normalized boundary-response parameter (the $\overline{\tau}$-functions) then selects comparison and associated model pairs for each connected component of $\Omega\setminus \mathrm{Max}(u)$, leading to the gradient and geometric estimates stated in Subsection~\ref{subsec:ComRes}.

\subsection{Comparison framework}\label{subsec:ComFram}

Unlike the planar ring-shaped setting, where the reference family is fixed a priori, here the comparison family must be built from the data of \eqref{DP}. We do so by introducing boundary-adapted \emph{model manifolds}, realized as warped products. This class includes the space forms and allows curvature information to be encoded through the warping function; it has been used, for instance, in eigenvalue comparisons on geodesic balls \cite{chengEigenvalueComparisonTheorems1975,hurtadoEstimatesFirstDirichlet2016,palmerFirstDirichletEigenvalue2024a}. In these Riemannian manifold structure the Laplace--Beltrami operator takes a particularly simple form, which we will use repeatedly; see also \cite{dipierroClassificationStableSolutions2019,berchioStabilityQualitativeProperties2014, castorinaRegularityStableSolutions2015, bisterzoSymmetrySolutionsSemilinear2023, manciniSemilinearEllipticEquation2009} for existence and regularity results for semilinear problems on related models.

\begin{definition}[Model Manifold]\label{def:model}
	Given $k \in \r $, we say that a Riemannian manifold 
	$(\mathcal M ^n, g)$ is a {\it model manifold} if it is a warped product of the 
	form: 
	$$(\mathcal M _k ^n (F),g_k) :=(I _k\times F,\,dr^2+s_k(r)^2 g_F) ,$$
	where the fiber  $F$ is an $(n-1)-$dimensional smooth compact Riemannian manifold and 
	$I_k := [0 ,\bar r _k)$, where $\bar r _ k := +\infty$ if $k \leq 0 $ and 
	$\bar r _k := \pi/\sqrt{k}$ if $k > 0$. The warping function is defined as,
	 \begin{equation}\label{eq:sk}
			 s_k (r) := 
	 	\left\{ \begin{matrix} \frac{\sinh (\sqrt{-k} r)}{\sqrt{-k}} & \text{ if } & k<0 \\[1mm] 
			r & \text{ if } & k=0 ,\\[1mm]
			 \frac{\sin (\sqrt{k} r)}{\sqrt{k}} & \text{ if } & k>0, \end{matrix}\right. 
	 \end{equation} 
\end{definition}
\begin{remark}[Conical singularity when $(F,g_F)\ncong\mathbb{S}^{n-1}$]\label{rem:conical-tip}
If $(F,g_F)\ncong(\mathbb S^{n-1},g_{\rm round})$, the warped product metric
$g_k=dr^2+s_k(r)^2g_F$ does not extend smoothly across $r=0$.
Indeed, since $s_k(r)=r+O(r^3)$ as $r\to0$, a smooth completion by adding a pole $o$
would force $g_k$ to be asymptotically Euclidean in geodesic polar coordinates,
which is equivalent to the link metric $g_F$ being the unit round metric.
In general, the metric completion has an isolated \emph{conical singularity} at $r=0$,
locally modeled on the cone $(0,\varepsilon)\times F$ with metric $dr^2+r^2g_F$; see \cite{SchroheConical}.
\end{remark}

For later use we introduce the following auxiliary notation, 
\[
\cot_k (r):= \frac{s_k' (r)}{s_k (r)}\quad \forall r \in (0, \bar r_k) \quad \textup{and} \quad \tan_k (r) = \frac{s_k (r)}{s_k ' (r)} ,  \quad \forall r \in [0, \bar r_k)\setminus \{\bar r_k/2\}.
\]
Since the metric has a warped product structure, the Laplace--Beltrami operator on a model manifold takes the following form,
\begin{equation}\label{eq:laplace-forms}
\Delta=\partial_r^2+(n-1)\, \cot _k(r)\,\partial_r+\frac{1}{s_k(r)^2}\,\Delta_{F},
\end{equation}

Having fixed the class of model manifolds, we now introduce the corresponding model solutions.
In the space--form case (that is, when \(F=\mathbb S^{n-1}\) with the round metric up to scaling), the natural reference class consists of radial solutions about a point \(o\), namely functions of the distance
\(r(p)=\mathrm{dist}(o,p)\).
More generally, on a model manifold \( (\mathcal M_k^n(F), g_k) \) with \( n \ge 2 \), we call \emph{radial} those solutions of \eqref{DP} that depend only on the base coordinate \(r\).

A direct computation using \eqref{eq:laplace-forms} shows that if \( u(p)=U(r(p)) \),
with \( U\in \mathcal{C}^2(\mathbb R) \), then \( u \) solves $\Delta u +f(u)=0$ if and only if \( U \) solves the ordinary differential equation
\begin{equation}\label{ODERadial}
U''(r) + (n-1)\,\cot_k(r)\,U'(r) + f\big(U(r)\big) = 0,
\end{equation}
for \( r\in(0,\bar r_k) \), where \( \bar r_k := \pi/\sqrt{k} \) if \( k>0 \)
and \( \bar r_k := +\infty \) otherwise. Thus, we can find solutions \( U \) of \eqref{ODERadial} satisfying the Cauchy conditions
\begin{equation}\label{CauchyData}
U(R)=M
\quad \text{and} \quad
U'(R)=0,
\end{equation}
for some \( R\in[0,\bar r_k) \) and \( M>0 \).  The parameter \( R \) will be referred to as the \emph{core radius}.
For \( R>0 \), the maximum of \( u \) is attained on the hypersurface \( \{r=R\} \),
whereas for \( R=0 \) the maximum is attained at the origin (the \emph{centered} case).

For the comparison we need profiles that stay positive between their first and last zeros, are strictly monotone away from $r=R$, and vanish at the boundary radii of the associated model domain, see Figure~\ref{fig:radprof}.  Since this is not automatic for a general locally Lipschitz $f$, we encode it in the notion of $k$--admissibility; see Definition~\ref{def:admissible-f}.

\begin{figure}
\centering
	\begin{tikzpicture}
  \begin{axis}[
    axis lines=middle,
    xlabel={$r$},
    ylabel={$U_{0,M,k,f}(r)$},
	xtick={0.1,1},
  xticklabels={{$R=0$},{$r_{+}(0,M,k,f)$}},
  ytick = {0.5},
  yticklabels={$M$},
    domain=0:4,
    samples=100,
    xmin=0, xmax=1.5,
    ymin=0, ymax=1,
    grid=none,
    width=5.5cm,
    height=5.5cm
  ]
    \addplot[thick,smooth, mark=none, color= blue,domain=0:1] { (1 - x^2)/2 };
	\addplot[only marks, mark=*, mark size=1.8pt, color = blue] coordinates {(1,0)};

	\addplot[color=blue,opacity=1,  line width=3pt] coordinates {(0.01,0) (1,0)};

	\addplot[only marks, mark=*, mark size=1.8pt, color = black] coordinates {(0,0.5)};
  \end{axis}
\end{tikzpicture}
		\begin{tikzpicture}
  \begin{axis}[
    axis lines=middle,
    xlabel={$r$},
    ylabel={$U_{R,M,k,f}(r)$},
	xtick={1.1,2,3.1},
  xticklabels={{$r_{-}(R,M,k,f)$},{$R$},{$r_{+}(R,M,k,f)$}},
    ytick = {1},
  yticklabels={$M$},
    domain=0:4,
    samples=100,
    xmin=0.9, xmax=3.5,
    ymin=0, ymax=1.5,
    grid=none,
    width=8.5cm,
    height=5.5cm
  ]

	\addplot[thick,smooth, mark=none, color = blue,domain=1.0:2] { (1 - 2*x^2)/4 + 4*ln(abs(x))};
	\addplot[thick,smooth, mark=none, color = blue,domain=2.0:3.1] { (1 - 2*x^2)/4 + 4*ln(abs(x))};

	\addplot[dashed, thick] coordinates {(2,0) (2,1)};

	\addplot[only marks, mark=*, mark size=1.8pt, color = black] coordinates {(2,1.03)};
	\addplot[only marks, mark=*, mark size=1.8pt, color = blue] coordinates {(1.092,0)};
	\addplot[color=blue,opacity=1,  line width=3pt] coordinates {(1.1,0) (2,0)};

		\addplot[color=blue,opacity=1,  line width=3pt] coordinates {(2,0) (3.1,0)};

	\addplot[only marks, mark=*, mark size=1.8pt, color = blue] coordinates {(3.092,0)};
  \end{axis}
\end{tikzpicture}
	\caption{Here we show the profile of a solution $U_{R,M,k,f}$ to \eqref{ODERadial} when $f$ is $k$--admissible.
\emph{Left:}  case $R=0$, where the solution $U_{0,M,k,f}$ is positive on $(0,r_{+}(0,M,k,f))$, vanishes at $r=r_{+}(0,M,k,f)$, and attains its maximum at the origin. 
\emph{Right:}  case $R>0$, where $U_{R,M,k,f}$ is supported on the interval $(r_{-}(R,M,k,f),r_{+}(R,M,k,f))$, vanishes at both boundary radii, and attains its maximum at $r=R$. }
\label{fig:radprof}
\end{figure}
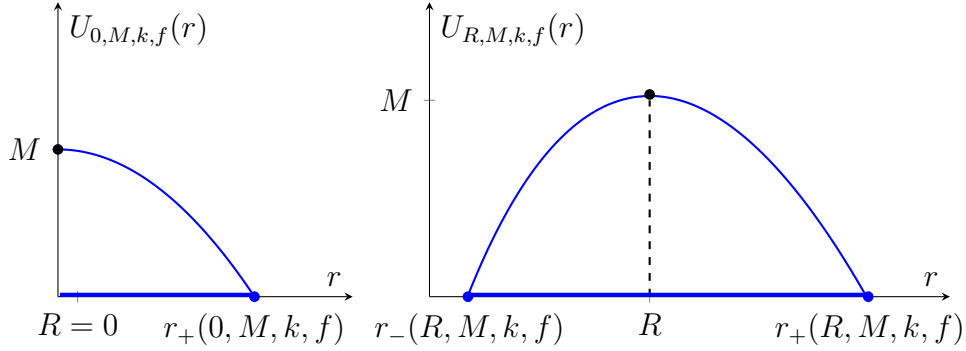

\begin{definition}[Admissible nonlinearities] \label{def:admissible-f}
Given $k \in \mathbb{R}$, a locally Lipschitz function $f\in\textup{Lip}_{loc}(\mathbb{R})$ is called $k-$\emph{admissible} if there exist non-empty connected intervals 
\[
0 \in \mathcal{R}_f \subset [0,\bar r_k)
\quad\text{and}\quad
\mathcal{I}_f \subset \r_+ \text{ with } 0 \in \partial \mathcal{I}_f
\]
such that, for every \((R,M)\in \mathcal R_f\times \mathcal I_f\), there exists a maximal solution \(U_{R,M,k,f}\) of \eqref{ODERadial}--\eqref{CauchyData} that is defined
on an interval that contains a closed interval \([r_-(R,M,k,f),r_+(R,M,k,f)]\) with
\[
0\leq r_-(R,M,k,f)\leq R<r_+(R,M,k,f)\le \bar r_k,
\]
and satisfies:
\begin{itemize}
\item If $R=0$, then $r_-(0,M,k,f) = 0 < r_+(0,M,k,f)$, 
\begin{itemize}
\item $U_{0,M,k,f} > 0$  on $ [0,r_+(0,M,k,f))$, 
\item $U'_{0,M,k,f} \neq 0 $  on $ (0,r_+(0,M,k,f))$, and 
\item $U_{0,M,k,f}(r_+(0,M,k,f)) = 0$.
\end{itemize}

\item If $R>0$, then $ 0 < r_-(R,M,k,f) < R < r_+(R,M,k,f) $, 
\begin{itemize}
\item $U_{R,M,k,f} > 0$  on $(r_-(R,M,k,f),r_+(R,M,k,f))$, 
\item $U'_{R,M,k,f} \neq 0$  on $(r_-(R,M,k,f),r_+(R,M,k,f)) \setminus \{R \}$, and 
\item $U_{R,M,k,f}(r_\pm(R,M,k,f)) = 0$.
\end{itemize}
\end{itemize}
\end{definition}
The family of $k$--admissible functions comprises a wide class of nonlinearities, for example the Helmholtz equation, studied in \cite{Ros}, and its affine version; the Lane--Emden equation $f(x)=x|x|^p$; or Allen--Cahn--type terms of the form $f(x)=x-x^p$, where $p \in(1,\frac{n+2}{n-2})$ in $\mathbb{S}^d$, for $n>1$, as studied in \cite{ruizOverdeterminedEllipticProblems2023a}.

 When $k>0$, the solutions to \eqref{ODERadial}-\eqref{CauchyData} enjoys the following symmetry: if \(U_{R,M,k,f}(r)\) is a solution with core radius \( R\), then the solution with core radius \(\bar r_k- R\) satisfies that $U_{R,M,k,f}(r)=U_{\bar r_k-R,M,k,f}(\bar r_k-r)$. Thus, the solution whose core radius is $R=\bar r_k/2$  is symmetric with respect to the line $r=R$.

Note that $k$--admissibility forces $f>0$ on $\mathcal I_f$. Otherwise, if there exists $M_0 \in \mathcal{I}_f$ with $f(M_0)\leq 0$, then \eqref{ODERadial}-\eqref{CauchyData} would imply that all functions $U_{R,M_0,k,f}$ have either a minimum or a saddle point at $r=R$, contradicting the properties stated above. 

In general, it is not clear what the most general conditions on $f$ are that ensure $k$--admissibility for the problem \eqref{ODERadial}-\eqref{CauchyData}. We will be able to guarantee $k$--admissibility of a locally Lipschitz function $f\in \operatorname{Lip}_{loc}(\mathbb{R})$ under the following conditions: 

\begin{refquote}{Standard Conditions}{Standard Conditions}
\quotemark{quote:StandardCondition}{standard conditions} 

There exists an interval $I_f \subset \r_+$ with $0 \in \partial I_f$ such that
\[
f(x)>0 \quad \textup{and} \quad f(x) \geq nk x + f(0), \quad \forall x \in I_f,
\]
where $f(0)>0$ if $k \leq 0$. 
\end{refquote}
The \nameref{quote:StandardCondition} take a particularly simple form when $k=0$, since they reduce to
$f(x)\ge f(0)$ on $I_f$, so $f$ is bounded from below by its value at the origin. Indeed, the function \( f(x) \equiv 0 \) is not \( 0 \)-admissible.
When $k<0$, the term $nkx$ is non-positive for $x\ge0$, hence the constraint
$f(x)\ge nkx+f(0)$ is weaker than $f(x)\ge f(0)$; in particular, it does \emph{not} force $f$ to be non-decreasing near $0$, but it still requires $f(x)>0$ on $I_f$.
In fact, the function \( f(x) = nkx \) is not \( k \)-admissible when \( k < 0 \).

The case $k>0$ (for instance on spherical geometries) is more restrictive, because $nkx$ is increasing and
the condition imposes a definite lower growth in $x$. For example, for an affine nonlinearity
$f(u)=\lambda u+\beta$ (an affine Helmholtz-type term), the inequality
$f(x)\ge nkx+f(0)$ is equivalent to $\lambda\ge nk$ (since $f(0)=\beta$), together with $f(x)>0$ on $I_f$.
Another standard example is the Liouville--Bratu--Gelfand--type nonlinearity $f(u)=K e^{2u}$ with $K\in \mathbb{R}$,
so the standard condition holds provided $K\ge nk$. The study of this problem in dimension 2 is linked to prescribing curvatures in the disk via conformal changes; see, for example, \cite{ruizConformalMetricsDisk2024, lopez-sorianoPrescribingCurvaturesDisk2025}.

For power-type nonlinearities, one should keep in mind the positivity requirement $f(x)>0$ on $I_f$.
For instance, $f(x)=x-x^p$ is non-negative for $x\in[0,1]$, hence it satisfies the standard conditions on an interval with $0$ in its boundary.
Similarly, for Brezis--Nirenberg-type terms of the form $f(x)=x^{2^*-1}+\lambda x$ (with $2^*=\frac{2n}{n-2}$ for $n\ge3$), the inequality near $0$ forces $\lambda\ge nk$.

The following result is proved in Appendix \ref{app:radial-ode}.
\begin{proposition}\label{prop_ODE1}
Let $k \in \r$ and $f \in \textup{Lip}_{loc} (\r)$ satisfying the \nameref{quote:StandardCondition}. Then $f$ is $k-$admissible with admissible sets $\mathcal{I}_f = I_f$ and $\mathcal{R}_f = [0,  \bar r_k)$.
\end{proposition}
\begin{remark}
The previous proposition  remains true if the \nameref{quote:StandardCondition} are replaced by a  more technical assumption: namely, that \( f \geq \tilde{f} \) on \( I_f \), where \( \tilde{f} \) admits a solution to \eqref{ODERadial} with maximum value \( M > 0 \) on a bounded interval, for every \( M \in I_f \). This can be seen by inspecting the proof in  Appendix~\ref{app:radial-ode}.
\end{remark}

Given a model manifold $(\mathcal{M}_k^n (F), g_k)$, assume $f\in \textup{Lip}_{loc}(\mathbb{R})$ is a $k-$admissible function. Let $U_{R,M,k,f}$ be the solution of \eqref{ODERadial}--\eqref{CauchyData} constructed above. The profile of $U_{R,M,k,f}$ depends on the geometric data, the dimension $n$, the curvature bound $k$, the nonlinearity $f$, and on the Cauchy data, $(R,M)$, but not on the fiber $F$. For each $(R, M ) \in \mathcal{R}_f \times \mathcal{I}_f$, define the domain
	\[
	\Omega_{R,M,k,f}:=
	\begin{cases}
		\{\,p \in \mathcal{M}_k^n (F) :  r (p)<r_+(0,M ,k,f)\,\}, & R=0,\\[2pt]
		\{\,p \in \mathcal{M}_k^n (F) : r_-(R,M,k,f)<r (p)<r_+(R,M,k,f)\,\}, & R \neq 0.
	\end{cases}
	\] and set
	\[
	u_{R,M,k,f}(p):=U_{R,M ,k,f} (r(p)),\qquad p\in \Omega_{R,M,k,f}.
	\]
	
Then \((\Omega_{R,M,k,f},u_{R,M,k,f})\) solves \eqref{DP} and on each  boundary component of $ \Omega_{R,M,k,f}$ the the norm of the gradient of $u_{R,M,k,f}$ is locally constant on each boundary component. 

The pairs \((\Omega_{R,M,k,f},u_{R,M,k,f})\) comprise all positive radial solutions of \eqref{DP} with constant Neumann datum in each connected component of  the model manifold and will be the reference objects for comparison. Moreover: 
\begin{itemize}
\item If $R =0$, $\Omega_{0,M,k,f}$ is called the \textit{centered domain}, and we set 
\[
\Gamma_{0,M,k,f} = u^{-1}_{0,M,k,f}(M) \text{ and }\Gamma_{0,M,k,f}^{+} = \{\,p \in \mathcal{M}_k^n (F) : r (p)=r_+(0,M ,k,f)\,\}. 
\]
We also denote by $\U_{0,M,k,f}^{+} \subset \Omega_{0,M,k,f} \setminus \Gamma_{0,M,k,f}$ the subdomain of $\Omega_{0,M,k,f}$ without maximum points and we set $\U_{0,M,k,f}^{-}=\emptyset$.  The boundary of $\U ^+ _{0,M,k,f}$ is connected if $F$ is a connected manifold.

\item If $R \in (0, \bar r_k)$, we denote 
\[
\Gamma_{R,M,k,f} = u^{-1}_{R,M,k,f}(M) \text{ and } \Gamma_{R,M,k,f}^{\pm} = \{\, p \in \mathcal{M}_k^n (F) : r (p)=r_\pm (R,M ,k,f)\,\}, 
\]and $\U_{R,M,k,f}^{\pm} \subset \Omega_{R,M,k,f} \setminus \Gamma_{R,M,k,f}$ the subdomains of $\Omega_{R,M,k,f}$ without maximum points such that $\textup{cl} (\U_{R,M,k,f}^{\pm}) \cap \partial \Omega_{R,M,k,f} = \Gamma_{R,M,k,f}^{\pm}$.
\end{itemize}

\begin{definition}[Model pairs]\label{def:model-pairs}
Given a model manifold $(\mathcal{M}_k^n (F), g_k)$, let $f\in \textup{Lip}_{loc}(\mathbb{R})$ be a $k-$admissible function. For \((R,M)\in \mathcal R_f\times \mathcal I_f\), define the sets $\U_{R,M,k,f}^{\pm}$ as above
and the corresponding functions
\[
u_{R,M,k,f}^\pm(p):=U_{R,M,k,f}(r(p)) \quad \text{ on } \U_{R,M,k,f}^\pm.
\]
Then, the pairs \((\U_{R,M,k,f}^\pm,u_{R,M,k,f}^\pm)\) are called 
\textup{model pairs} and we denote by ${\rm Model}_{M,k,f}^{\pm}$ the family of 
model pairs $(\mathcal{U}^\pm_{R,M,k,f},u^\pm_{R,M,k,f})$ obtained by varying the 
\emph{core radius} $R\in\mathcal{R}_f$ of the solution. 

We call \((\U_{0,M,k,f}^+,u_{0,M,k,f}^+)\) the \textit{centered} model pair of ${\rm Model}_{M,k,f}^+$. Also, we write $\textup{Model}_{M,k,f} = {\rm Model}_{M,k,f}^{+} \sqcup {\rm Model}_{M,k,f}^{-}$ for the complete family of model pairs.

\end{definition}

\subsection{Model \texorpdfstring{$\overline\tau$}{tau}-functions}\label{subsec:model-tau}

To track how the choice of initial data affects the boundary behavior of these solutions, we introduce the $\overline{\tau}$-functions. These functions encode the boundary response of the radial profiles and provide a convenient way to compare their boundary behavior with that of the centered solution.

Given a model manifold $(\mathcal{M}_k^n (F), g_k)$, let $f\in \textup{Lip}_{loc}(\mathbb{R})$ be a $k-$admissible function. Since \(f\circ u_{0,M,k,f}^+\) is positive in \(\U_{0,M,k,f}^+\) for each \(M\in \mathcal I_f\), Hopf's lemma yields that the centered model solution \((\U_{0,M,k,f}^+,u_{0,M,k,f}^+)\) determines a positive constant
\begin{equation}\label{eq:normalization-constant}
\mathsf c(M,k,f):=\bigl|\nabla u_{0,M,k,f}^+\bigr|^2_{|_{\Gamma_{0,M,k,f}^+}}
=
U'_{0,M,k,f}\bigl(r_+(0,M,k,f)\bigr)^2>0,
\end{equation}
which depends only on \((M,k,f)\) (and not on the choice of fiber \(F\)). We define the \emph{model \(\overline\tau\)-functions} by
\begin{equation}\label{eq:tau-model}
\overline\tau_{M,k,f}^\pm(R):=
\frac{U'_{R,M,k,f}\bigl(r_\pm(R,M,k,f)\bigr)^2}
     {U'_{0,M,k,f}\bigl(r_+(0,M,k,f)\bigr)^2}
=
\frac{|\nabla u_{R,M,k,f}^\pm|^2_{|_{\Gamma_{R,M,k,f}^\pm}}}
     {\mathsf c(M,k,f)},
\qquad R\in[0, \bar r_k).
\end{equation}
These ratios encode how the boundary gradient of a model solution compares to that of the centered solution, and they will be used to parametrize comparison pairs.  For us it will be useful to introduce the following notation for the lower and upper bounds of the functions $\overline\tau_{M,k,f}^{\pm}$,
\[
 \tau_{k,f}^-(M):=\inf_{R\in[0,\bar r_k)} \overline\tau_{M,k,f}^-(R),\quad
 \tau_{k,f}^+(M):=\sup_{R\in[0,\bar r_k)}\overline\tau_{M,k,f}^+(R)
\]
and
\[
\tau_{k,f}^0(M):=\inf_{R\in[0,\bar r_k)}\overline\tau_{M,k,f}^+(R).
\]

In \cite{ABM,ABBM,EMa}, the monotonicity of the $\overline{\tau}$-functions is clear due to the nature of the nonlinearity. However, in the present setting the monotonicity of the model $\overline{\tau}$-functions is not evident, and even determining the range of their images is not immediate. Thanks to the maximum principle we can state some results under some restrictions on the nonlinearity. 

In the special case $f(x)=f_k(x):=nkx+1$, the problem \eqref{ODERadial}--\eqref{CauchyData} corresponds to Serrin's equation.   In this case, it is easy to check that $\overline\tau_{M,k,f}^-(R)\to+\infty$ as $R \to 0^+$. We prove it in Proposition \ref{prop:infitytau}, in Appendix \ref{app:radial-ode}. In this appendix, we also provide the proof of the following result:
\begin{proposition}\label{prop:behav}
	Let $f$ be a function satisfying the \nameref{quote:StandardCondition}.  Then, for any $M \in \mathcal{I}_f$, $\overline\tau_{M,k,f}^-(R)\to+\infty$ as $R\to0^+$. Moreover, one has $ \tau_{k,f}^+(M)\le  \tau_{k,f}^-(M)<+\infty$ if $k\le0$, and $ \tau_{k,f}^+(M)\ge \tau_{k,f}^-(M)$ if $k>0$.
\end{proposition}

Finally, note that under an additional condition on $f$ one can show monotonicity of $\overline\tau_{M,k,f}^\pm$. Again, we pospose the proof of this result to Appendix \ref{app:radial-ode}.
\begin{proposition}\label{prop:ODE2}
	Let $f\in \mathcal C^1(\mathbb{R})$ satisfy the \nameref{quote:StandardCondition}, and suppose there exists $\bar M>0$ such that $(0,\bar M)\subset\mathcal{I}_f$ and
	\begin{equation}\label{conditionf2}
	f(x)\ge x f'(x)\quad\text{for all }0<x<\bar M.
	\end{equation} 
	Then for any $M\in(0,\bar M)$ we have $ \tau_{k,f}^0(M)= \overline\tau_{M,k,f}^+ (0)=1$, and $\overline\tau_{M,k,f}^-(R)$ is strictly decreasing while $\overline\tau_{M,k,f}^+(R)$ is strictly increasing on $(0,\bar r_k)$.
\end{proposition}
For future use, we define the \emph{admissible set} and \emph{gap} associated to \((M,k,f)\) as
\begin{equation}\label{eq:adm-gap}
{\mathrm{Adm}}_{M,k,f}:={\rm Im}(\overline\tau_{M,k,f}^-)\cup {\rm Im}(\overline\tau_{M,k,f}^+),
\quad
{\mathrm{ Gap}}_{M,k,f}:=\r_+^*\setminus ({\rm Adm}_{M,k,f} \cup [0, \tau_{k,f}^0(M)]).
\end{equation}
Note that, by the continuous dependence with respect to the initial data of solutions to differential equations with locally Lipschitz coefficients (see for example \cite[Theorem 2.21]{EV}), it follows that ${\rm Im}(\overline\tau_{M,k,f}^-)$ and ${\rm Im}(\overline\tau_{M,k,f}^+)$ are connected intervals. Therefore, ${\rm Adm}_{M,k,f}$ is an interval or the union of two disjoint intervals, and then ${\rm Gap}_{M,k,f}$ can be empty, a point or the union of intervals with non-empty interior. As we shall see later, these three options can actually occur (see  Remark \ref{rem:gap-not-trivial} and Corollary \ref{coro_existenceComparisonPairs}).

Proposition \ref{prop:behav} implies that, when $f$ satisfies the \nameref{quote:StandardCondition}, $\mathrm{Gap}_{M,k,f}= \emptyset$ when $k>0$ and consists of a bounded interval with extrema $ \tau_{k,f}^+(M)$ and $\tau_{k,f}^-(M)$ if $k\leq 0$. Here, we consider that $[a,a] = \set{a}$ for any $a \in \r$.

\begin{remark}\label{rem:gap-not-trivial}
	Proposition \ref{prop:ODE2} and the symmetry of the solutions to \eqref{ODERadial}-\eqref{CauchyData} when $k>0$ imply that, when $f$ satisfies condition \eqref{conditionf2} and  \nameref{quote:StandardCondition},  we get that
	\begin{equation}\label{eq:gap}
 \tau_{k,f}^\pm (M) =	\displaystyle\lim_{R \to \bar r_k/2} \overline\tau_{M,k,f}^\pm (R), \quad \forall M \in \mathcal{I}_f.
	\end{equation}
	When $k=0$ and $f$ is the Serrin's problem non-linearity this result also holds.
	This is already noted in \cite{ABM, ABBM},  and it also holds that $\overline \tau^+_{0,f_0}(M)=\overline \tau^-_{0,f_0}(M) $.
	This implies that ${\rm Gap}_{M,0,f_0}$  consists of a single value, $\{\overline \tau^+_{0,f_0}(M) \}$. 
	However, when $k<0$ using \eqref{eq:gap} we can show that ${\rm Gap}_{M,k,f_k}$ consists of a closed interval with non-empty interior.
	
	In fact, given any $M \in (0,1/n)$, write $V_{R,M,-1} (r)$ for the solution to \eqref{ODERadial}-\eqref{CauchyData} with $k=-1$ and $f(x)=f_{-1}(x) = -nx+1$. Then, by doing the change of variable $r =R+s$, it is straightforward to show that $V_{R,M,-1} $ converges (to first order) to the function 
	\[
	V_{\infty,M}(s) = 1-\frac{1}{(n+1)(M+1)}\left(  e^{-ns}+n\ e^{s} \right)
	\]
	as $R \to +\infty$. Thus, if $s_- <0< s_+$ denotes the two roots of $V_{\infty,M}$, we can check that
	\[
	\begin{split}
		\lim_{R\rightarrow \infty} -V'_{R,M,-1}(r_+) & \approx \frac{n}{(n+1)(M+1)}\left(e^{s_+}-e^{-ns_+} \right)  \\
		&< \frac{n}{(n+1)(M+1)}\left(e^{-ns_-}-e^{s_-} \right) \approx \lim_{R\to \infty} V'_{R,M,-1}(r_-).
	\end{split}
	\]
	By denoting $U_{R,M,-1,f_{-1}} = u_R$  and $r_{\pm}(R,M,-1,f_{-1}) = r_{\pm} (R)$ with $M =\cosh (3)-1$, in Figure \ref{fig:gap} we obtain numerically the sizes of the interval $\textup{Gap}_{M,-1,f_{-1}}$ for different values of $n$.
\end{remark}
\begin{figure}[h]
	\centering
	\includegraphics[width=0.4\textwidth]{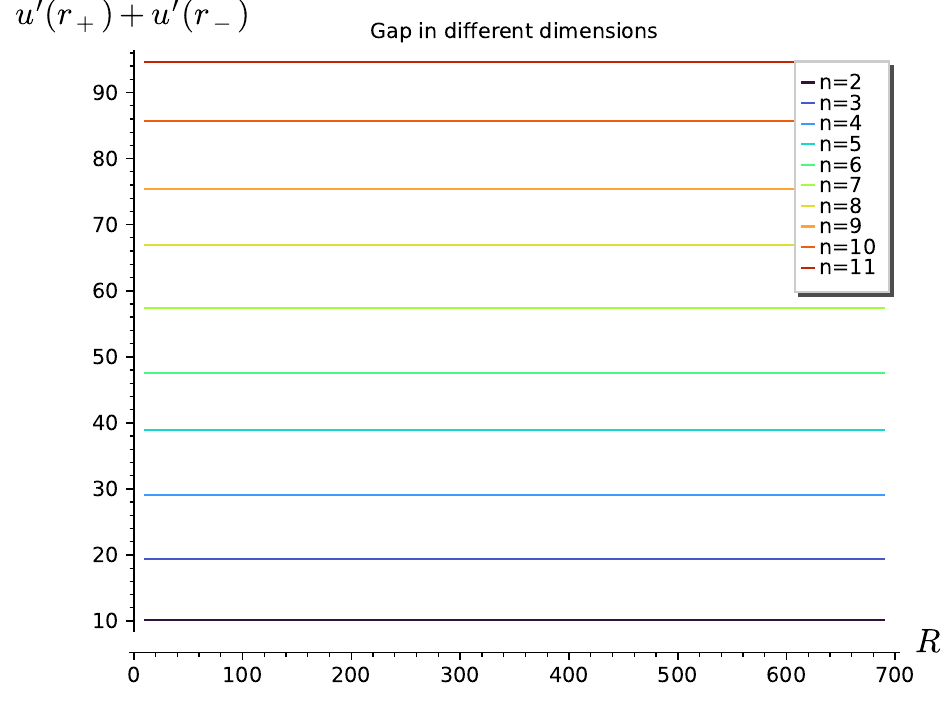}
	\caption{In this graph we picture the function $s(R)=| u_R' (r_- (R))+ u_R' (r_+(R)) |$, where $u_R$ is a solution for the Serrin's problem as stated in Remark \ref{rem:gap-not-trivial} with Cauchy data \eqref{CauchyData},  for different dimensions $n$, which are represented with different colors. Note that these functions are so close to be constants for $R$ big that they look as straight lines in the scale presented in this picture. Observe also that $\mp u_R'(r_{\pm} (R)) =  c(M,-1,f_{-1}) \cdot \tau_{M,-1,f_{-1}}^\pm(R)$, so each graph have the horizontal line of height $c(M,-1,f_{-1}) \cdot \textup{Length} ({\rm Gap}_{M,-1,f_{-1}})$ as asymptote.}\label{fig:gap}
\end{figure}

\subsection{Construction of the comparison framework}\label{subsec:ComMeth}
Let \((\mathcal M,g)\) be an \(n\)-dimensional complete Riemannian manifold with a Ricci lower bound $\operatorname{Ric}_g \ge (n-1)k\,g$, and let \(\Omega\subset \mathcal M\) be a bounded \(\mathcal{C}^2\)-domain. Fix \(f\in \mathrm{Lip}_{\mathrm{loc}}(\mathbb R)\) that is \(k\)--admissible in the sense of Definition~\ref{def:admissible-f}, and consider a solution \((\Omega,u)\) of~\eqref{DP}.

The problem~\eqref{DP} is invariant under isometries: if \(\mathcal I\in \mathrm{Iso}(\mathcal M)\), then \((\mathcal I(\Omega),\,u\circ \mathcal I^{-1})\) also solves~\eqref{DP}. We say that \((\Omega,u)\) and \((\tilde\Omega,\tilde u)\) are \emph{congruent}, denoted \((\Omega,u)\equiv (\tilde\Omega,\tilde u)\), if there exists \(\Phi\in  \mathrm{Iso}(\mathcal M)\) such that \(\Phi(\Omega)=\tilde\Omega\) and \(\tilde u = u\circ \Phi^{-1}\) in \(\tilde\Omega\).

\paragraph{Step 1. Choosing a connected component.}
Working up to congruence, fix \((\Omega,u)\) and set \(M:=u_{\textup{max}}\in \mathcal I_f\).
Let \(\mathcal U\) be a connected component of \(\Omega\setminus \mathrm{Max}(u)\).
Since $u$ is a solution to \eqref{DP} and $f$ is $k$--admissible, it follows from the maximum principle that $\partial\Omega\cap \textup{cl}({\mathcal U})\neq\emptyset$. 

\begin{definition}\label{def:tau-U}
For each connected component \(\Gamma\in\pi_0(\partial\Omega)\), define
\[
\overline\tau(\Gamma):=\max_{p\in\Gamma}\frac{|\nabla u(p)|^2}{\mathsf c(M,k,f)}.
\]
Then, for each $\mathcal U \in \pi _0 \left( \Omega\setminus\mathrm{Max}(u)\right)$, set 
\[
\overline\tau(\mathcal U):=\max\bigl\{\overline \tau(\Gamma):\Gamma\in\pi_0(\partial\Omega\cap\textup{cl} (\U))\bigr\}.
\]
For the whole domain we define, 
\[
\overline\tau(\Omega):=\max\bigl\{\overline \tau(\mathcal{U}):\mathcal{U} \in\pi_0(\Omega\setminus\mathrm{Max}(u))\bigr\}.
\]
\end{definition}

\paragraph{Step 2. Construction of the model manifold.}
Choose a boundary component \(\Gamma_0\in\pi_0(\partial\Omega\cap \textup{cl}({\mathcal U}))\) such that
\(\overline\tau(\Gamma_0)=\overline\tau(\mathcal U)\). We then work on the model manifold whose fiber is \(\Gamma_0\) equipped with the metric induced by $g$ on $\Gamma_0$, denoted by \(\bar g\):
\[
(\mathcal M _k ^n (\Gamma_0), g_k) :=(I _k\times \Gamma_0,\,dr^2+s_k(r)^2 \bar g).
\]
This choice fixes the reference geometry for the comparison on \(\mathcal U\).

\paragraph{Step 3. Comparison criteria within the model family.}
Within the model pairs on \(\mathcal M_k^n(\Gamma_0)\) with maximum value \(M\), we single out those whose normalized boundary response matches (or dominates) \(\overline\tau(\mathcal U)\).

\begin{definition}[Comparison pair and associated model pair]\label{def:comparison-pair}
We say that a model pair
\[
(\bar\U ,\bar u) :=(\mathcal{U}_{R,M,k,f}^\diamond ,u_{R,M,k,f}^\diamond)\in {\rm Model}_{M,k,f}^{\diamond}
, \qquad \diamond\in\{-,+\},
\]
in $\mathcal M_k^n(\Gamma_0)$ is a \emph{comparison pair} for \((\mathcal U,u)\) if
\begin{equation}\label{eq:comparison-inequality}
\overline\tau(\mathcal U)\le \overline\tau(\bar\U)=\overline\tau_{M,k,f}^\diamond(R),
\end{equation}
and an \emph{associated model pair} if equality holds. In either case, we call \(\bar u\) a \emph{comparison function} for \(u\) in \(\mathcal U\).
\end{definition}

In Subsection \ref{subsubsec:AboutModelPairs} we discuss existence and non-uniqueness of comparison pairs and associated model pairs, and we record consequences of the \nameref{quote:StandardCondition}.

\paragraph{Step 4. Construction of the comparison tool.}
Fix a comparison pair \((\bar \U,\bar u)\) for \((\mathcal U,u)\). Denote by \(\bar U\) the radial profile generating \(\bar u\), by \(\bar R\) its core radius, and write \(\bar r_\pm:=r_\pm(\bar R,M,k,f)\).
Following \cite[Subsection 5.4]{EMa2}, define \(G:[0,M]\times[\bar r_-,\bar r_+]\to\mathbb R\) by
\[
G(s,r):=s-\bar U(r).
\]
Then \(\partial_r G(s,r)=0\) if and only if \(\bar U'(r)=0\), i.e.\ \(r=\bar R\) (recall Definition \ref{def:admissible-f}). By the Implicit Function Theorem, there exist two \(\mathcal{C}^2\)-functions
\begin{equation}\label{chiFunctions}
\chi^-:[0,M)\to[\bar r_-,\bar R],
\qquad
\chi^+:[0,M)\to[\bar R,\bar r_+],
\end{equation}
such that \(G(s,\chi^\pm(s))=0\) for all \(s\in[0,M)\), i.e.\ \(\bar U(\chi^\pm(s))=s\). Equivalently, \(\chi^\pm\) is the inverse of \(\bar U\) on the corresponding monotonicity interval. We use these inverses to define a pseudo-radial map on \(\mathcal U\), which will allow us to compare the geometry of level sets in Section~\ref{sectionGradient}.

\begin{definition}[Pseudo-radial function]\label{def:pseudo-radial}
Let \((\bar \U,\bar u)\) be a comparison pair for \((\mathcal U,u)\) with sign \(\diamond\in\{-,+\}\).
The \emph{pseudo-radial function} associated to \((\mathcal U,u)\) and \((\bar \U,\bar u)\) is
\[
\Psi: \mathcal U\to \chi ^\diamond (0,M) \subset [\bar r_-,\bar r_+],
\qquad
\Psi(p):=\chi^\diamond(u(p)).
\]
\end{definition}

\begin{remark}\label{rem:psi-basic}
From the definition of the function $\chi^{\diamond}$, it is clear that it extends continuously to $[0,M]$. Furthermore, since \(u\in \mathcal{C}^2(\Omega)\), the function \(\Psi\) is \(\mathcal{C}^2\) on \(\mathcal U\) away from the critical set of \(u\), and \(\Psi\) is Lipschitz on compact subsets of \(\mathcal U\). Moreover, by construction \(\bar U\circ\Psi=u\) on \(\mathcal U\).
\end{remark}

At this point the role of $k$--admissibility becomes clear: it ensures that for each prescribed maximum value $M$ and core radius $R$, the corresponding radial Cauchy solution is positive between its first and last zeros, strictly monotone away from the maximum set, and vanishes at the boundary radii. These properties allow us to invert the profile on each branch and to define the pseudo-radial map \(\Psi\) for a general solution.

\subsubsection{Existence and non-uniqueness of associated model pairs}\label{subsubsec:AboutModelPairs}
We record a few remarks about comparison/associated model pairs and about the admissible range of the parameter $\overline\tau$.

\paragraph{On non-uniqueness.}
An associated model pair in Definition~\ref{def:comparison-pair} need not be unique: if the map \(R\mapsto \overline\tau_{M,k,f}^\diamond(R)\) is not injective, distinct core radii can produce the same $\overline\tau$-value. This does not affect the arguments below, which only require the existence of at least one comparison (or associated) pair.

\paragraph{Existence issues.}
By definition, an associated model pair exists if and only if \(\overline\tau(\mathcal U)\in{\rm Adm}_{M,k,f}\). For a general $k$--admissible $f$ we do not know whether this always holds. Nevertheless, comparison pairs exist under mild lower bounds on $f$, and associated model pairs exist in many cases of interest; see Lemma~\ref{lemma_Adm} below.

As a consequence of Proposition~\ref{prop:behav}, we obtain the following existence statement.

\begin{corollary}\label{coro_existenceComparisonPairs}
	Let $f \in \textup{Lip}_{loc} (\r)$ be a function that satisfies the \nameref{quote:StandardCondition}. Suppose that $\Omega \subset \mathcal{M}$ is a bounded domain and that $u$ solves \eqref{DP}. Then for every $\U \in \pi_0 (\Omega \setminus \textup{Max}(u))$, there exists at least one comparison pair for $(\U,u)$. In the particular case of $k > 0$, for every $\overline{\tau}(\mathcal{U}) > \tau^0_{k,f}(M)$ there is an associated model pair.
\end{corollary}

We will also use the following notion of regular point.

\begin{definition}\label{defin_RegularPoint}
Let $A \subset \mathcal{M}$ and let $l \in \set{1, \dots, n}$. We say that $p \in A$ is an $l$-\textup{dimensional regular point} if $A$ is an $l$-dimensional $\mathcal{C}^1$-submanifold in a neighbourhood of $p$.
\end{definition}

Under assumptions stronger than the \nameref{quote:StandardCondition}, the comparison method becomes effective in the sense that regularity information on $\mathrm{Max}(u)\cap \textup{cl}(\U)$ forces $\overline\tau(\U)$ to lie in the admissible range. This generalizes \cite[Proposition 4.3]{EMa2} (which in turn is a version of \cite[Theorem 3.3]{EMa}).

\begin{lemma}\label{lemma_Adm}
	Let $f\in \mathcal{C}^{m} (\r)$ with $m:= \max\{n-2,3\}$ be a $k-$admissible function with $f(0)>0$ if $k \leq 0$, and assume
	\begin{equation}\label{conditionf}
			f'(x)\ge nk \quad \text{for all } x\in\mathcal{I}_f.
	\end{equation}
	Then:
	\begin{enumerate}
	\item If $\operatorname{Max}(u)\cap \textup{cl}(\U)$ contains an $(n-1)$-dimensional regular point, then $\overline\tau(\mathcal{U}) > \overline \tau^0_{k,f}(M)$.
	\item If $(\mathcal{M},g)=\mathbb{M}^n(k)$, then $\overline\tau (\mathcal U) \notin \mathrm{Gap}_{M,k,f}$.
	\end{enumerate}
	In particular, if $(\mathcal{M},g)=\mathbb{M}^n(k)$ and $\mathrm{Max}(u)\cap \textup{cl} (\U)$ contains a smooth hypersurface, then there exists an associated model pair for $(\mathcal{U},u)$ with positive core radius.
\end{lemma}

We postpone the proof of this result to Section~\ref{subsec_Curvature}, where \eqref{conditionf} will also be crucial for obtaining the gradient estimate away from $\mathrm{Max}(u)$.

To clarify how the hypothesis packages on $f$ feed into the properties of the $\overline\tau$-functions and the existence of comparison/associated model pairs, Figure~\ref{fig:conditionsdependency} summarizes the logical dependencies.
\begin{figure}[th]
\centering
\begin{adjustbox}{width=0.65\linewidth}

\begin{tikzpicture}[
  font=\small,
  >=Latex,
  xgap/.store in=\xgap, xgap=2.8cm,
  ygap/.store in=\ygap, ygap=1.4cm,
  box/.style={draw, rounded corners=3pt, align=center, inner sep=5pt,
              minimum width=2.6cm, minimum height=0.95cm},
  orange/.style={box, very thick, draw=orange!85!black, fill=orange!12},
  blue/.style={box, very thick, draw=blue!75!black, fill=blue!10},
  purple/.style={box, very thick, draw=violet!70!black, fill=violet!10},
  green/.style={box, very thick, draw=green!60!black, fill=green!10},
  red/.style={box, very thick, draw=red!75!black, fill=red!10},
  req/.style={->, thick},
  suff/.style={->, thick, dashed},
  node distance=12mm and 16mm
]

\node[purple] (Std) at (0,0) { \nameref{quote:StandardCondition}};

\node[blue, left=8mm of Std]     (Kad)     {$k$-admissible\\ (Definition \ref{def:admissible-f})};
\node[blue, above=8mm of Kad]     (GradAss) {$f'(x)\ge nk$\\\eqref{conditionf}};
\node[blue, below=8mm of Kad] (TauCond) {$f(x)\ge x f'(x)$\\ \eqref{conditionf2}};
\draw[suff] (Std) -- (Kad);


\node[orange, right=16mm of Std] (P25) {Blow-up and ordering of $\overline \tau$\\(Proposition \ref{prop:behav})};
\node[green, above=4mm of P25] (L214)  {Existence of \\ an associated model pair\\(Lemma \ref{lemma_Adm})};
\node[green, below=4mm of P25] (C212) {Existence of comparison \\ pairs (Corollary \ref{coro_existenceComparisonPairs})};
\node[orange, below=4mm of C212] (P26) {Monotonicity and\\ normalization of $\overline \tau$\\(Proposition \ref{prop:ODE2})};


\draw[req] (Std) -- (P25);
\draw[req] (Std) -- (C212);
\draw[req] (Std) to[bend right=22] (P26);
\draw[req] (TauCond) -- (P26);

\draw[req] (Kad) to[bend left=12] (L214);
\draw[req] (GradAss) -- (L214);

\end{tikzpicture}

\end{adjustbox}
\caption{  The diagram summarizes the dependencies between the results previously stated and the conditions over the non-linearity. Arrows indicate requirements (the arrows point to the conditions necessary to obtain that result); dashed arrows indicate sufficient conditions. Orange boxes correspond to properties of the $\overline\tau$-functions, green boxes to existence statements for comparison/model pairs, and purple/blue boxes to hypothesis on $f$.}
\label{fig:conditionsdependency}
\end{figure}
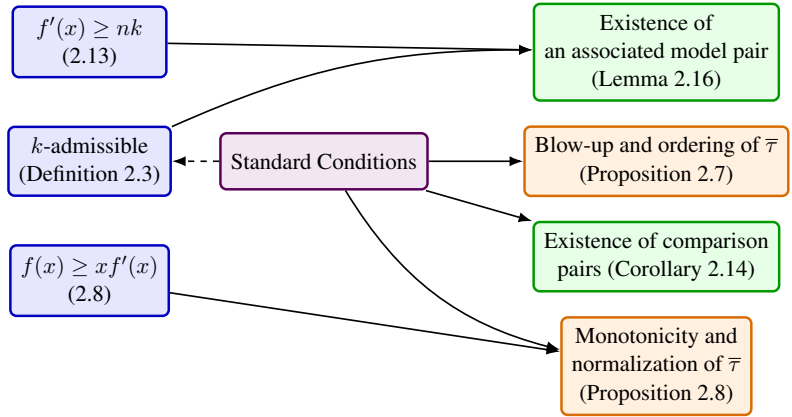 

\subsection{Comparison results}\label{subsec:ComRes}
In this section we are going to state the main comparison results. 
Let $(\Omega,u)$ solve the semilinear Dirichlet problem~\eqref{DP} on an $n$-dimensional Riemannian manifold $(\mathcal M,g)$ satisfying the Ricci lower bound \eqref{eq:Ric}, and let $f\in\mathrm{Lip}_{\mathrm{loc}}(\mathbb R)$ be $k$--admissible in the sense of Definition~\ref{def:admissible-f}. Fix a component 
\[
\mathcal U\in \pi_0(\Omega\setminus \mathrm{Max}(u))
\quad\text{with}\quad
M:=u_{\max}\in\mathcal I_f,
\]
and fix the model manifold constructed in Subsection~\ref{subsec:ComMeth}, namely $(\mathcal M _k ^n (\Gamma_0),g_k)$. Let $(\bar\U,\bar u)$ be a comparison pair for $(\mathcal U,u)$ in the sense of Definition~\ref{def:comparison-pair}. Denote by $\bar U$ the radial profile generating $\bar u$, by $\bar R$ its core radius, and set $\bar r_{\pm}:=r_{\pm}(\bar R,M,k,f)$. Let $\Psi$ be the associated pseudo-radial function (Definition~\ref{def:pseudo-radial}), so that $u=\bar U\circ\Psi$ on $\mathcal U$. Throughout this subsection we use the shorthand
\begin{equation}\label{definW}
W:=|\nabla u|^2 \quad\text{and}\quad \bar W:=\bar U'(\Psi)^2 \quad \text{on }\mathcal U.
\end{equation}

The main input is a pointwise gradient comparison (Theorem~\ref{theo:Gradient}). From it we derive geometric bounds for the boundary and top level sets, an isoperimetric-type inequality (Theorem~\ref{theo:Isoperimetric}), and a quantitative hot-spots estimate (Theorem~\ref{theo:HotSpots}). Figure~\ref{fig:reqsec3} summarizes the dependency structure among the hypothesis packages on $f$ and the resulting statements.

\begin{figure}[th]
\centering
\begin{adjustbox}{width=0.65\linewidth}
\begin{tikzpicture}[
  font=\small,
  >=Latex,
  xgap/.store in=\xgap, xgap=2.8cm,
  ygap/.store in=\ygap, ygap=1.4cm,
  box/.style={draw, rounded corners=3pt, align=center, inner sep=5pt,
              minimum width=2.6cm, minimum height=0.95cm},
  orange/.style={box, very thick, draw=orange!85!black, fill=orange!12},
  blue/.style={box, very thick, draw=blue!75!black, fill=blue!10},
  purple/.style={box, very thick, draw=violet!70!black, fill=violet!10},
  green/.style={box, very thick, draw=green!60!black, fill=green!10},
  red/.style={box, very thick, draw=red!75!black, fill=red!10},
  req/.style={->, thick},
  suff/.style={->, thick, dashed},
  node distance=12mm and 16mm
]


\node[blue]     (Kad)   at (0,0)   {$k$-admissible};
\node[blue, above=8mm of Kad]     (GradAss) {$f'(x)\ge nk$\\\eqref{conditionf}};
\node[blue, above=8mm of GradAss]     (CM) {$f\in \mathcal{C}^m$\\ $m = \max\{n-2,3\}$};
\node[blue, below=8mm of Kad] (Pinch)   {$f(x)\le nkx+1$};

\node[red, right=15mm of Kad] (GradEst) {Gradient estimates\\ (Theorem \ref{theo:Gradient})};
\node[red, above=8mm of GradEst] (CurvTop) {Bound for the\\Mean curv.\\ of top level set\\ (Proposition \ref{prop_CurvatureMaxSets})};

\node[red, below=4mm of GradEst]  (Hot) {Location of hot-spots\\ (Theorem \ref{theo:HotSpots})};

\node[red, right=8mm of CurvTop] (BoundArea) {Bound for the\\area of top\\level set\\ (Proposition \ref{prop_AreaLevelSet})};

\node[red, below=8mm of BoundArea] (Iso) {Isoperimetric type\\ inequality\\ (Theorem \ref{theo:Isoperimetric})};

\node[green, below=4mm of Iso] (MeanCurv) {Bound for the\\Mean curv.\\of zero level set\\ (Proposition \ref{prop_curvatureZeroSets})};


\draw[req] (Kad) -- (GradEst);
\draw[req] (GradAss) -- (GradEst);
\draw[req] (GradEst) to[bend left=10]  (MeanCurv) ;

\draw[req] (GradEst) -- (Hot);
\draw[req] (Pinch) --(Hot);

\draw[req] (GradEst) --  (CurvTop);
\draw[req] (CM) -- (CurvTop);

\draw[req] (GradEst) --  (BoundArea);
\draw[req] (CM) to[bend left=15] (BoundArea);
\draw[req]  (BoundArea) -- (Iso);

\end{tikzpicture}

\end{adjustbox}

\caption{The diagram summarizes the dependencies among the hypotheses on $f$, the theorems, and the resulting geometric bounds.  Arrows indicate requirements (the arrows point to the conditions necessary to obtain that result). On blue we represent the condition over $f$, on red we represent the results that we need a comparison pair and in green the ones that we need an associated model pair.}
\label{fig:reqsec3}
\end{figure}
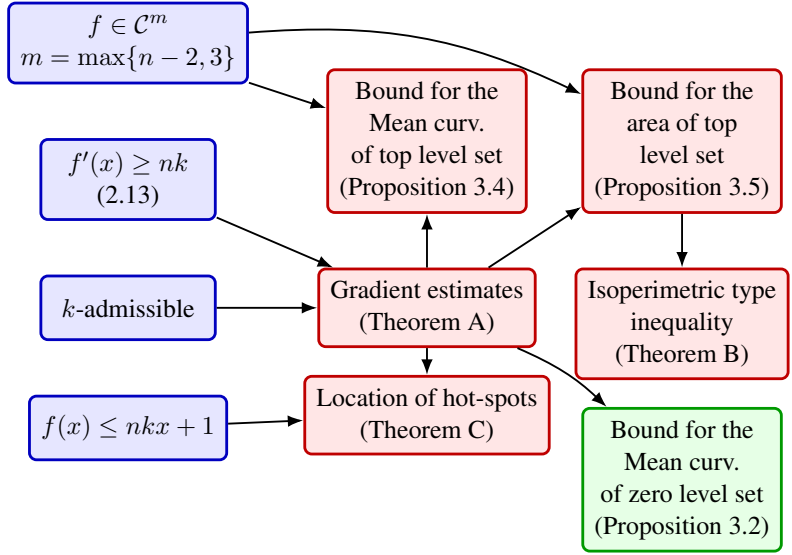

\begin{lettertheorem}[Gradient estimates] \label{theo:Gradient}
Let $f\in \mathcal C^{1}(\mathbb R)$ be a $k-$admissible function 
satisfying condition \eqref{conditionf}:
\begin{equation*}
f'(x)\ge nk \quad \text{for all } x\in\mathcal I_{f}.
\end{equation*}

 Then, for any comparison pair 
$(\bar{\U},\bar u)$ in the sense of Definition \ref{def:comparison-pair}, we have
\[
W(p)\le \bar W(p) \quad \text{for all } p\in\mathcal U.
\]
Furthermore, if equality holds at any point $p\in\mathcal U$, 
then $\mathrm{Ric}(\nabla u,\nabla u) = (n-1)k\,|\nabla u|^{2}$ on $\mathcal U$, 
$(\mathcal{M}^n,g)$ is isometric to $(\mathcal{M}^n_k(\Gamma_0),g_k)$ in $\U$ and moreover, $(\mathcal U,u)$ is congruent to $(\bar{\U},\bar u)$.
\end{lettertheorem}

\begin{remark}
If \( (\mathcal{U}, u) \equiv (\bar \U, \bar u) \), then the level sets of \(u\) in \( \mathcal{U} \) are compact, embedded totally umbilical hypersurfaces, hence strongly constrain the geometry of \( \mathcal{U} \). In particular, if \( (\mathcal{M}, g) = \mathbb{M}^n(k) \) is a space form, then the classification of totally umbilical hypersurfaces forces \( \Gamma = \mathbb{S}^{n-1} \) and \( g_0 = g_{\mathbb{S}^{n-1}} \). Consequently, \( \mathcal{U} \) is either an annular or a metric ball domain and \(u\) is radially symmetric.
\end{remark}

The question of optimal assumptions on $f$ is discussed in Subsection~\ref{ConditionsThmA}.

As consequences of Theorem~\ref{theo:Gradient}, we obtain curvature bounds for $\partial\Omega$ (Proposition~\ref{prop_curvatureZeroSets}) and for the top level set (Proposition~\ref{prop_CurvatureMaxSets}), as well as an area bound for the top level set (Proposition~\ref{prop_AreaLevelSet}). In higher dimensions, these curvature estimates are naturally weaker than in the surface case: since the PDE controls $\Delta u$ the comparison yields effective control of the \emph{mean} curvature of level sets but not of the full principal curvature spectrum when $n\geq 3$ (see \cite{ABBM}).

Combining the area bound with the coarea formula yields the following isoperimetric-type inequality. For any $s>0$, denote by $\mathcal{H}^s$ the $s$-dimensional Haussdorff meausure associated to the metric $g$ of $\mathcal{M}$, and by $\mathcal{H}_k^s$ that associated to the metric $g_k$ of the model manifold $\mathcal{M}_k^n(\Gamma_0)$.

The proof of this result is given in Section~\ref{sec:IsoIneq}.

\begin{lettertheorem}[Isoperimetric inequality]
\label{theo:Isoperimetric}
		Let $f\in \mathcal C^{m}(\mathbb R)$, with $m = \max\{n-2,3\}$, be a $k-$admissible function satisfying  the condition \eqref{conditionf}. Assume that the $(n-1)$-dimensional
		part of $\textup{cl}(\U) \cap \textup{Max}(u)$, denoted by  $\Gamma_{M}$, consists  of a (possibly disconnected) 
		Lipschitz-continuous hypersurface.  Hence, for any comparison pair $ (\bar \U, \bar u)$,
		\begin{equation}\label{volumeBound}
\frac{\mathcal{H}^n (\U)}{\mathcal{H}^{n-1} (\Gamma_{M})} \geq \left\{ \begin{matrix}
\frac{\mathcal{H}_k^n (\mathcal{U}^+_{\bar R, M,k,f})}{\mathcal{H}_k^{n-1} (\Gamma_{\bar R, M,k,f})}
& \text{ if } & \bar \U \in \textup{Model}_{M,k,f}^+ \\[6mm]
\frac{\mathcal{H}_k^n (\mathcal{U}^-_{\bar R, M,k,f})}{\mathcal{H}_k^{n-1} (\Gamma_{\bar R, M,k,f})}  
& \text{ if } & \bar \U \in \textup{Model}_{M,k,f}^-.
\end{matrix}\right.
		\end{equation}
	 Furthermore, equality holds in \eqref{volumeBound} if, and only if $(\U,u) \equiv (\bar \U, \bar u)$.
\end{lettertheorem}

We conclude with a quantitative localization estimate for the maximum set (hot spots), extending the sharp bounds obtained for Serrin-type equations in \cite{MP,ABM}. The proof is given in Section~\ref{sec:hotspots}.

\begin{lettertheorem}[Location of the hot-spots]
\label{theo:HotSpots}
Let $f \in \mathcal{C}^1 (\r)$  be a $k-$admissible function satisfying  the condition \eqref{conditionf} and 
\begin{equation*}
f(x) \leq f_k(x)=nkx+1 \qquad
\end{equation*}
Assume that $\Omega$ is contained in a convex domain of $\mathcal M$ and that $\overline{\tau} (\Omega) = \overline{\tau} (\U)$.  
If $k>0$ suppose that $\overline r_+\leq \bar r_k/2$.
  Let $(\bar \U,\bar u)$ be an comparison pair to $(\mathcal{U},u)$ with  core radius $\bar R$.
Then,

\begin{equation}\label{ineqHotSpotsNormalized}
\frac{\mathrm{dist}(p,\partial\Omega)}
     {\tan_{k}(r_{\Omega})/n}
\begin{cases}
\geq \displaystyle \frac{\bar r_{+}-\bar R}
         {\sqrt{\frac{2}{n}M + k\,M^{2}}}
& \text{if } (\bar \U, \bar u) \in {\rm Model}_{M,k,f}^+, \\[6mm]
 >\displaystyle \frac{\bar R-\bar r_-}
         {\sqrt{\frac{2}{n}M + k\,M^{2}}}
& \text{if } (\bar \U, \bar u) \in {\rm Model}_{M,k,f}^-,
\end{cases}
\qquad \text{for all } p\in\operatorname{Max}(u).
\end{equation}
where $r_{\Omega} = \max \left\{ \textup{dist} (p, \partial \Omega) ~\colon~ p \in \Omega \right\}$.
 Moreover, if equality holds at a point of $\Omega$, then 
$(\mathcal M,g)$ is isometric to the space form $\mathbb M^{n}(k)$, 
$f(x)=nkx+1$, and $(\Omega,u)$ is a ball solution of Serrin's problem.
\end{lettertheorem}

The additional hypotheses in Theorem~\ref{theo:HotSpots} have distinct roles. The upper bound $f(x)\le nkx+1$ ensures that the Serrin's ball profile provides a subsolution, yielding a lower bound for $u$ in terms of the distance to $\partial\Omega$ by the maximum principle. Condition \eqref{conditionf} (i.e.\ $f'(x)\ge nk$) is used for the gradient comparison in Theorem~\ref{theo:Gradient}.  Moreover, the condition over $\mathcal{U}$, $\overline{\tau}(\mathcal{U})= \overline{\tau}(\Omega)$ is to be able to apply Lemma \ref{lemma_distanceToBoundary} to obtain a lower bound for $u(p)$.

\section{Proof of the Comparison Results}\label{sectionGradient}

Let \((\mathcal M^n,g)\) be a complete Riemannian manifold with the lower Ricci bound \eqref{eq:Ric}, and let \(f\in{\rm Lip}_{loc}(\mathbb R)\) be $k$-admissible. Fix a solution $(\Omega,u)$ of \eqref{DP} and choose a connected component
\[
\mathcal U\in \pi_0(\Omega\setminus \mathrm{Max}(u))
\qquad\text{with}\qquad
\overline{\tau}(\mathcal{U})= \overline{\tau}(\Omega).
\]
In this section we prove the gradient comparison (Theorem~\ref{theo:Gradient}) and then derive the geometric estimates stated in Section~\ref{sec:ComMethRes}. Technical computations are deferred to Appendix~\ref{appendixGradient}.

\begin{proof}[Proof of Theorem \ref{theo:Gradient}]
By Lemma \ref{lemma_ineq} (Appendix~\ref{appendixGradient}), the function
\[
F_\beta = \left(\frac{s_k (\Psi)}{\bar U ' (\Psi)}\right)^{2 \frac{n-1}{n}} (W-\bar W)
\]
where $s_k$ is given by \eqref{eq:sk} and $F_\beta$ satisfies the elliptic inequality
\[
\Delta F_\beta - 2 \frac{n-1}{n} \lambda (\Psi) \pscalar{\nabla F_\beta}{\nabla u}- 2 \frac{n-1}{n}\frac{\abs{\nabla u}^2}{\bar U ' (\Psi)^2} \mu (\Psi) F_\beta \geq 0,
\]
with $\mu$ being non-negative in $\mathcal U$. Moreover, $F_\beta$ extends continuously to $\textup{cl}(\mathcal U)$ and satisfies $ F_\beta \le 0 $ on $\partial\mathcal U$. Indeed, on $\partial\Omega\cap \textup{cl}(\mathcal U)$ we have $u=0$, hence $\Psi$ is constant and $\bar W$ equals the corresponding model boundary gradient; the comparison-pair condition \eqref{eq:comparison-inequality} implies $W\le \bar W$ there, and therefore $F_\beta\le 0$ on $\partial\Omega\cap \textup{cl}(\mathcal U)$. On the remaining portion $\textup{cl}(\mathcal U)\cap \mathrm{Max}(u)$, one has $W=0$ and $\bar W=0$; in addition, the weight in the definition of $F_\beta$ is chosen precisely so that $F_\beta$ has a finite continuous trace and vanishes on this set (see the proof of Lemma~\ref{lemma_ineq}).

Applying the maximum principle, we get that $W\le \bar W$ on $\mathcal U$. If equality holds at one point, then the strong maximum principle forces $F_\beta\equiv 0$, and therefore $W = \bar W$ on $\mathcal U$. Inspecting the inequality in Lemma~\ref{lemma_ineq} then yields
\[
\mathrm{Ric}(\nabla u,\nabla u)=(n-1)k|\nabla u|^2 \, \text{ on }\mathcal U.
\]
Moreover, by the proof of Lemma~\ref{lemma_ineq} one also obtains the identity at each level set:
\[
\nabla^2 u = -\lambda(\Psi)\,du\otimes du + \frac{1}{n}\Bigl(\lambda(\Psi)|\nabla u|^2 -f(u)\Bigr) g_0,
\]
where $\lambda$ is given by \eqref{eq:deflambda} and $g_0$ is the induced metric on $\Gamma_0$ (recall Definition \ref{def:comparison-pair}). Hence, Brinkmann’s Theorem (see \cite[Theorem 4.3.3]{Pe}) shows that the metric on $\U$ must then split as
\[
g = d\Psi^2 + s_k(\Psi)^2 \bar g \quad \text{on } \U,
\]
so $(\U,u)$ is congruent to the model pair $(\bar\U,\bar u)$, as claimed.
\end{proof}

\begin{remark}\label{remark_Gradient}
If \( \overline\tau(\mathcal{U}) =\overline\tau (\bar \U)= 1 \), extending \( \Psi \) to \( \operatorname{Max}(u) \cap \operatorname{cl}(\mathcal{U}) \), we obtain
\[
\nabla^2 S_k(\Psi) = s_k'(\Psi)\, g,
\]
with \( S_k \circ \Psi = 1 \) and \( \nabla (S_k \circ \Psi) = 0 \) on this set. It then follows again from \cite[Theorem 4.3.3]{Pe} that \( g_0 = g_{\mathbb{S}^{n-1}} \) and \( \Gamma = \mathbb{S}^n \). This, in turn, implies that
\[
\operatorname{Max}(u) = \{o\}, \qquad (\mathcal{M}, g) = \mathbb{M}^n(k),
\]
 \( \Omega \) is a ball, and \( u \) is radially symmetric with respect to \( o \).
\end{remark}

\subsubsection*{About the conditions of Theorem  \ref{theo:Gradient}}\label{ConditionsThmA}

As mentioned in Subsection \ref{subsec:ComRes}, we do not claim that the hypotheses in Theorem~\ref{theo:Gradient} are optimal. For instance, the torsional rigidity problem on $\mathbb{S}^2$ treated in \cite{EMa2} does not satisfy these assumptions, yet gradient estimates can still be obtained in that setting.

On $\mathbb{S}^3$ with the round metric, one can also write explicit radial solutions for affine Helmholtz nonlinearities $f(x)=\lambda x+\beta$, with $\lambda \in (-1,0)\cup(0,\pi)$ (excluding singular cases) and $\beta \in \mathbb{R}^+$ chosen so that $f$ is $k$-admissible:
\begin{equation*}
u(r)=\frac{\csc (r) \sin (R) (\lambda M+\beta) \cos \left(\sqrt{\lambda+1} (r-R)\right)+\frac{\csc (r) \cos (R) (\lambda M+\beta) \sin \left(\sqrt{\lambda +1} (r-R)\right)}{\sqrt{\lambda +1}}-\beta}{\lambda}.
\end{equation*}
For these explicit solutions, we observe that for certain initial data and suitable choices of $(\lambda,\beta)$ the quantity $\mu(r)$ in Lemma \ref{lemma_ineq} is positive, as in the right panel of Figure \ref{fig:expThmA}. Moreover, there are parameter ranges $(\lambda,\beta)$ for which $\mu$ in Lemma \ref{lemma_ineq} remains positive independently of the initial data, as illustrated in the left panel of Figure \ref{fig:expThmA}. However, in other cases, such as Allen--Cahn-type nonlinearities, one finds that if $\bar U(r)$ is a solution on $\mathbb{S}^n$ for $n\geq 2$, then at the boundary radii $\bar r_\pm$ one has $\mu(\bar r_\pm)=1-nk<0$, independently of the initial conditions.

\begin{figure}
\centering
\includegraphics[
    width=0.8\textwidth,
    clip
]{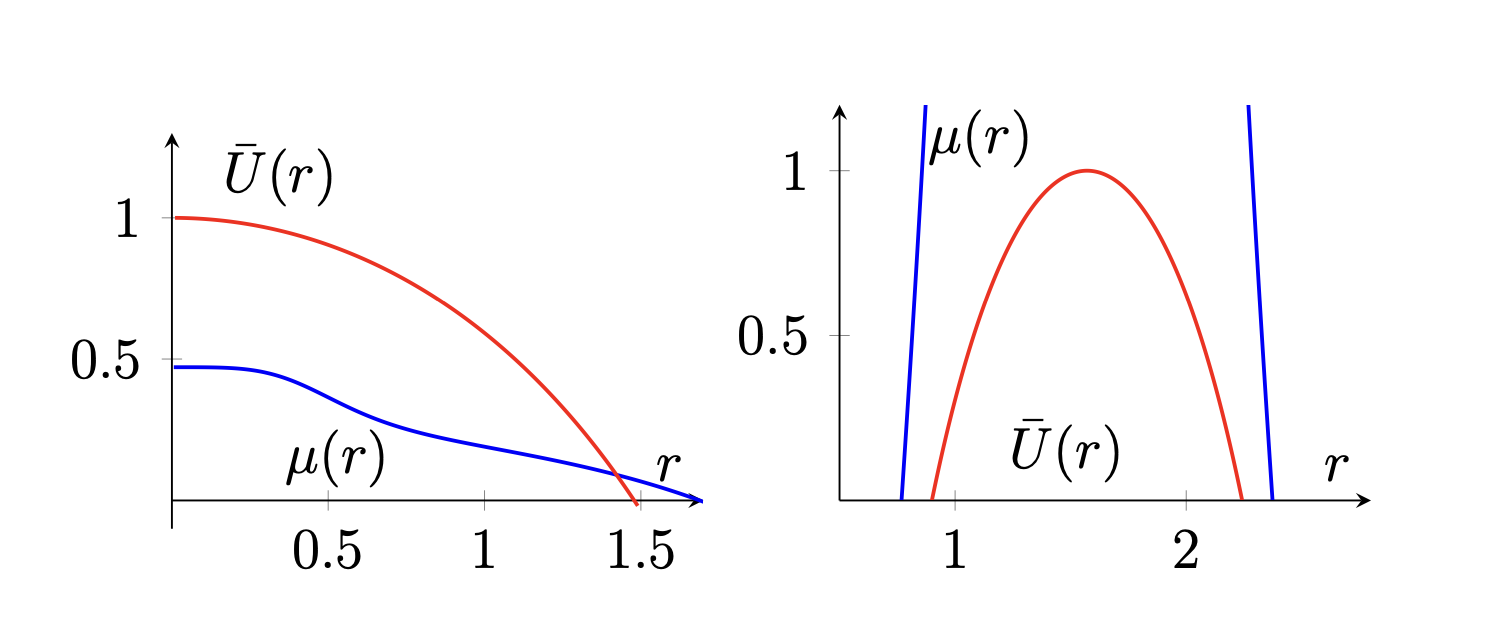}

\caption{Comparison of the radial profile $\bar U$ (red) and the function $\mu(r)$ defined in \eqref{eq:defmu} (blue) for affine Helmholtz equations on the three-dimensional sphere that do not satisfy condition \eqref{conditionf2}.
\emph{Left:} Let $f(x)= \frac{-1}{4}x +\frac{5}{2}$ and consider the radial solution with maximum $1$. For this $f$ we compute numerically that $\mu(r)$ is positive independently of the initial conditions. \emph{Right:} Let $f(x)= x +2.9$ and consider the solution with maximum $1$ at radius $\frac{\pi}{2}$. For this affine term we verified (with Sage) that $\mu(r)$ is positive only for certain initial conditions. }\label{fig:expThmA}
\end{figure}

\subsection{Curvature estimates}\label{subsec_Curvature}

In this subsection we extend the curvature estimates in 
\cite[Subsection~4.2]{EMa2} to higher dimensions and prove 
Lemma~\ref{lemma_Adm}. Throughout, we work under the assumptions of 
Theorem~\ref{theo:Gradient}.

\begin{proposition}\label{prop_curvatureZeroSets} 
Let $f$ be a $k$-admissible function satisfying the conditions \eqref{conditionf2} and \eqref{conditionf}.  Let $(\bar \U, \bar u)$ be an associated model pair to $(\U, u)$ in the sense of Definition \ref{def:comparison-pair}, and consider $p \in \partial \Omega$ such that
	\begin{equation}\label{condNabla}
		\abs{\nabla u}^2 (p) = \max_{x\in \partial \Omega \cap \textup{cl}(\U)} \abs{\nabla u}^2(x).
	\end{equation}
	Let $\bar R \in \mathcal{R}_f$ be the core radius of $(\bar \U, \bar u)$. Then, it holds
	\begin{equation*}
		H (p) \leq \cot_k (\bar r_+) \quad \textup{if } (\bar \U, \bar u) \in {\rm Model}_{M,k,f}^+ \quad \textit{and} \quad H (p) \leq -\cot_k (\bar r_-) \quad \textup{if } (\bar \U, \bar u) \in {\rm Model}_{M,k,f}^-,
	\end{equation*}
	where $\bar{r}_- $ and $\bar{r}_+$ are defined as the zeros of the solution
	 to \eqref{ODERadial}-\eqref{CauchyData} defining $(\bar \U, \bar u)$ and $H(p)$ is computed with respect to the inner orientation to $\U$.

\end{proposition}

\begin{proof}
Let $p \in \partial \Omega \cap \textup{cl}(\U)$ satisfy \eqref{condNabla}, and let $\gamma:[0,\varepsilon)\to\mathcal M$ be the unit-speed geodesic with $\gamma(0)=p$ and $\gamma'(0)=\nabla u/\abs{\nabla u}(p)$. As in \cite[Proposition 5.1]{EMa2}, for $t$ small one has
\[
W(\gamma(t)) = W(p) + 2\Big((n-1)\sqrt{W}(p)\,H(p)-f(0)\Big)\sqrt{W}(p)\,t + O(t^2),
\]
and
\[
\bar W(\gamma(t)) = \bar W(p) + 2\Big(\mp (n-1)\cot_k(\bar r_\pm)\sqrt{\bar W}(p)-f(0)\Big)\sqrt{\bar W}(p)\,t + O(t^2),
\]
where we take $\bar r_+$ when $(\bar \U,\bar u)\in {\rm Model}_{M,k,f}^+$ and $\bar r_-$ otherwise. Since \eqref{condNabla} implies $W(p)=\bar W(p)$, the claim follows by comparing the first-order terms and using Theorem~\ref{theo:Gradient}.
\end{proof}

\begin{remark}
The conclusion of Proposition~\ref{prop_curvatureZeroSets} requires an associated model pair: only in this case do we have $W(p)=\bar W(p)$ at a boundary point satisfying \eqref{condNabla}.
\end{remark}

\begin{proposition}\label{prop_CurvatureMaxSets}
Let $f\in \mathcal{C}^{m} (\r)$ with $m:= \max\{n-2,3\}$ be a $k$-admissible function satisfying \eqref{conditionf}. Let $(\bar \U, \bar u)$ be a comparison pair with core radius $\bar R >0$, and let $p \in \textup{cl}(\U) \cap \textup{Max}(u)$ be a $(n-1)$-dimensional regular point. Then $\textup{Max}(u)$ is a $\mathcal{C}^{m+2}$-hypersurface in a neighbourhood of $p$, and
\begin{equation*}
H (p) \leq -\cot_k(\bar{R}) \quad \textup{if } (\bar \U, \bar u) \in {\rm Model}_{M,k,f}^+  
\quad \textup{and} \quad 
H (p)\leq \cot_k(\bar{R})\quad \textup{if } (\bar \U, \bar u) \in {\rm Model}_{M,k,f}^-,
\end{equation*}
where $H(p)$ is computed with respect to the inner orientation of $\U$.
\end{proposition}

\begin{proof}
Since $u=0$ along $\partial \Omega$ and $u>0$ in $\Omega$, we have $p\in\Omega$. Since $p$ is a regular point, there exists a neighbourhood $\mathcal{V}\subset\Omega$ of $p$ such that
\(
\Sigma:=\mathcal{V}\cap \textup{Max}(u)
\)
is an embedded, two-sided $\mathcal{C}^1$ hypersurface dividing $\mathcal{V}\setminus\Sigma$ into two connected components $\mathcal{V}_+$ and $\mathcal{V}_-$, with $\mathcal{V}_+\subset\U$. Since $f\in \mathcal{C}^m$, elliptic regularity gives $u\in \mathcal{C}^{m+2}$ in $\Omega$, and \cite[Theorem~2]{BanyagaHurtubise04} implies that $\Sigma$ is $\mathcal{C}^{m+2}$ in a (possibly smaller) neighbourhood of $p$.

Denote again this neighbourhood by $\mathcal{V}$ and define the signed distance function
\[
s(x):=\left\{\begin{matrix}
+\textup{dist}(x,\Sigma) & \textup{if } x\in\mathcal{V}_+,\\[2mm]
-\textup{dist}(x,\Sigma) & \textup{if } x\in\mathcal{V}_-.
\end{matrix}\right.
\]
Since $u$ is at least $\mathcal{C}^5$ in $\Omega$, we can use the fourth-order expansion in \cite[Theorem 3.1]{Chr} and reproduce the argument of \cite[Proposition 5.2]{EMa2} to obtain, around $p\in\Sigma$,
\begin{equation}\label{eq:expansion1}
W = f(M)^2 s^2 \left( 1+ (n-1)H(p)\, s \right)+ O(s^4),
\end{equation}
\begin{equation}\label{eq:expansion2}
\bar{W} = f (M)^2 s^2 \left(1+ (n-1) \left(\frac{H (p)}{3} \pm \frac{2 \cot_k(\bar R)}{3}\right)s\right)+ O (s^4),
\end{equation}
where $H(p)$ is computed with respect to the normal pointing into $\mathcal{V}_+$, and in \eqref{eq:expansion2} we take the positive sign when $(\bar \U, \bar u)\in {\rm Model}_{M,k,f}^+$ and the negative sign otherwise. The claim then follows as in Proposition~\ref{prop_curvatureZeroSets}, using that $W\le \bar W$ in $\mathcal{V}_+$ by Theorem~\ref{theo:Gradient}.
\end{proof}

As a consequence of Proposition~\ref{prop_CurvatureMaxSets}, we can now prove Lemma~\ref{lemma_Adm}.

\begin{proof}[Proof of Lemma \ref{lemma_Adm}]\label{proof:lemmaADM}
We first prove the \emph{item 1}. The argument is the same as in \cite[Proposition 5.3]{EMa2}, but we include it for completeness. Assume that there exists a $(n-1)$-dimensional regular point \(p\in \mathrm{Max}(u)\cap \bar{\mathcal U}\) such that $u$ is $\mathcal{C}^{m+2}$, $m+2\geq 5$, near $p$. Arguing by contradiction, suppose that \(\overline{\tau}(\mathcal U)\le \overline\tau_{M,k,f}(0)\). Then Proposition~\ref{prop_CurvatureMaxSets} implies that \(\mathrm{Max}(u)\) is $\mathcal{C}^{m+2}$ near $p$ and that its mean curvature satisfies
\[
H(p)\leq -\cot_{k}(R)\, \text{ for all } R\in(0,\bar r_{k}),
\]
since every pair in $\operatorname{Model}^+_{M,k,f}$ is a comparison pair. Letting $R \to 0^+$ yields that \(H(p)\) is unbounded, a contradiction. Therefore, $\overline{\tau}(\mathcal U)>\tau^0_{k,f}(M)=\inf_{R\in[0,\bar r_k)}\overline\tau_{M,k,f}^+(R)$.

\medskip
We now prove the \emph{item 2}. Let $(\mathcal{M},g) = \mathbb{M}^n (k)$ with $k \leq 0$. We show that
\[
\overline{\tau} (\U) \notin \mathrm{Gap}_{M,k,f}.
\]
Again, we argue by contradiction.

Since $f\in \mathcal{C}^{m}$ where $m:= \max\{n-2,3\}$ satisfies the \nameref{quote:StandardCondition} and $M \in \mathcal{I}_f$, the maximum principle ensures that $\textup{cl}(\mathcal{U}) \cap \partial \Omega \neq \emptyset$. Choose a connected component $\Gamma \subset \textup{cl}(\mathcal{U}) \cap \partial \Omega$. Because $\Gamma$ is part of a level set of a $\mathcal{C}^{m+2}$ function and $\Omega$ is bounded, $\Gamma$ is a properly embedded, compact $\mathcal{C}^{m+2}$ hypersurface in $\mathbb{M}^n(k)$. By the Jordan--Brouwer Separation Theorem, $\mathbb{M}^n(k) \setminus \Gamma = \mathcal{V}_1 \cup \mathcal{V}_2$, where $\mathcal{V}_1$ is open and bounded. Thus either $\mathcal{U} \subset \mathcal{V}_1$ or $\mathcal{U} \subset \mathcal{V}_2$.

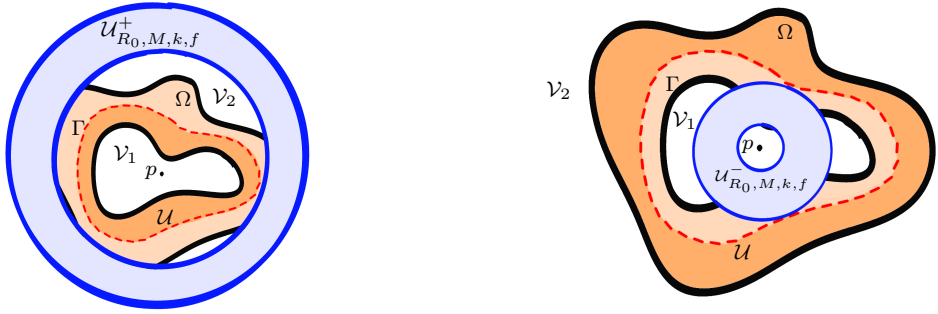
\begin{figure}[h!]
\centering
\begin{minipage}[t]{0.41\textwidth}
\vspace{40pt}
\centering
\begin{tikzpicture}[baseline=(current bounding box.north)]
  \node[anchor=south west, inner sep=0] (img) at (0,0)
    {\includegraphics[width=\linewidth]{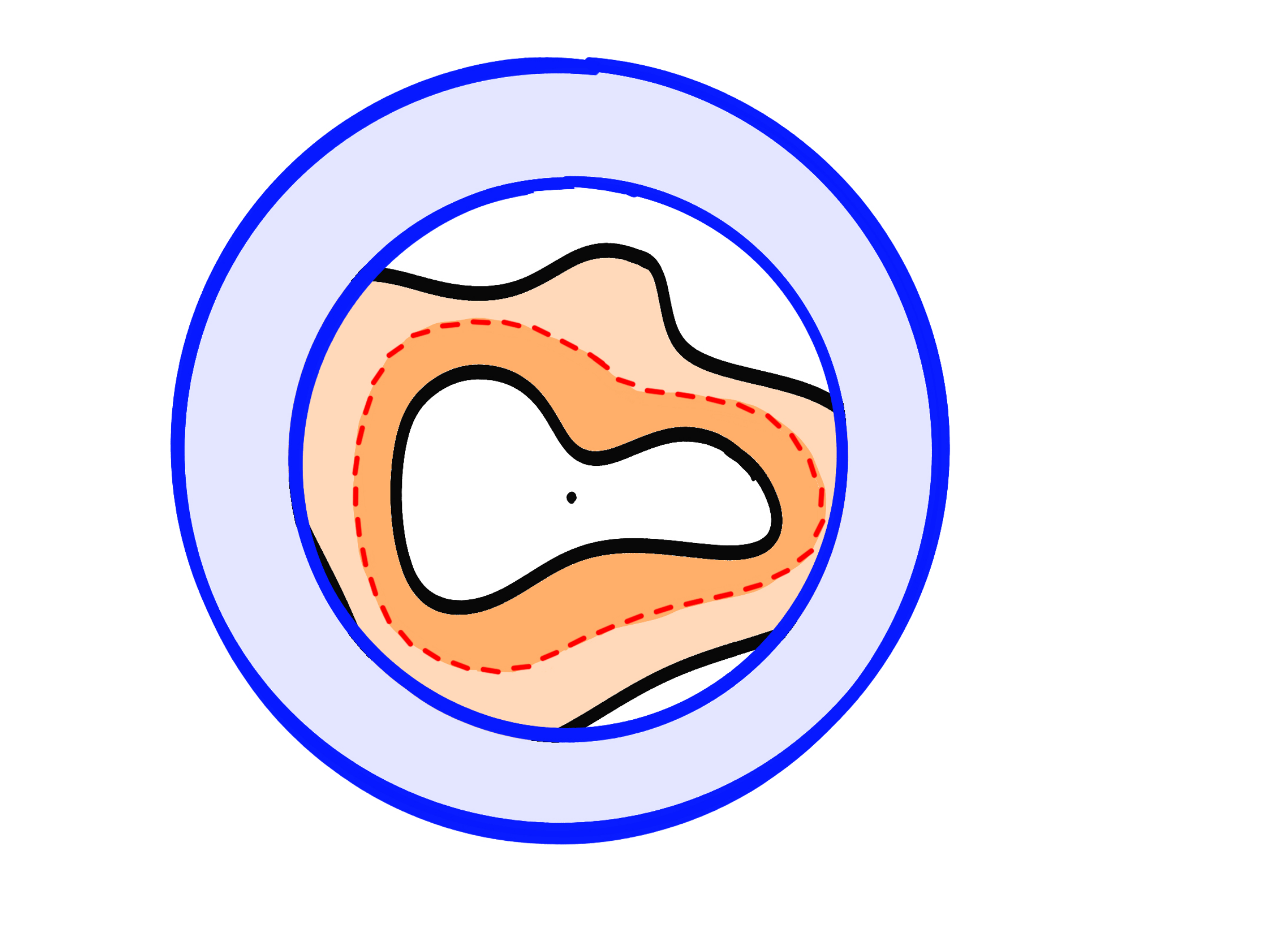}};
    
  \begin{scope}[x={(img.south east)}, y={(img.north west)}]
    \node[anchor=south west, font=\scriptsize] at (0.53,0.63) {$\mathcal{V}_2$};
    \node[ font=\scriptsize] at (0.43,0.86) {$\mathcal{U}^+_{R_0,M,k,f}$};
    \node[anchor=west, font=\scriptsize] at (0.245,0.6) {$\Gamma$};
    \node[anchor=west, font=\scriptsize] at (0.455,0.68) {$\Omega$};
    \node[anchor=west, font=\scriptsize] at (0.33,0.53) {$\mathcal{V}_1$};
    \node[anchor=west, font=\scriptsize] at (0.395,0.48) {$p$};
    \node[anchor=west, font=\scriptsize] at (0.42,0.36) {$\mathcal{U}$};
  \end{scope}
\end{tikzpicture}
\end{minipage}
\hfill
\begin{minipage}[t]{0.58\textwidth}
\vspace{0pt}

\centering
\begin{tikzpicture}
  \node[anchor=south west, inner sep=0] (img) at (0,0)
    {\includegraphics[width=\linewidth]{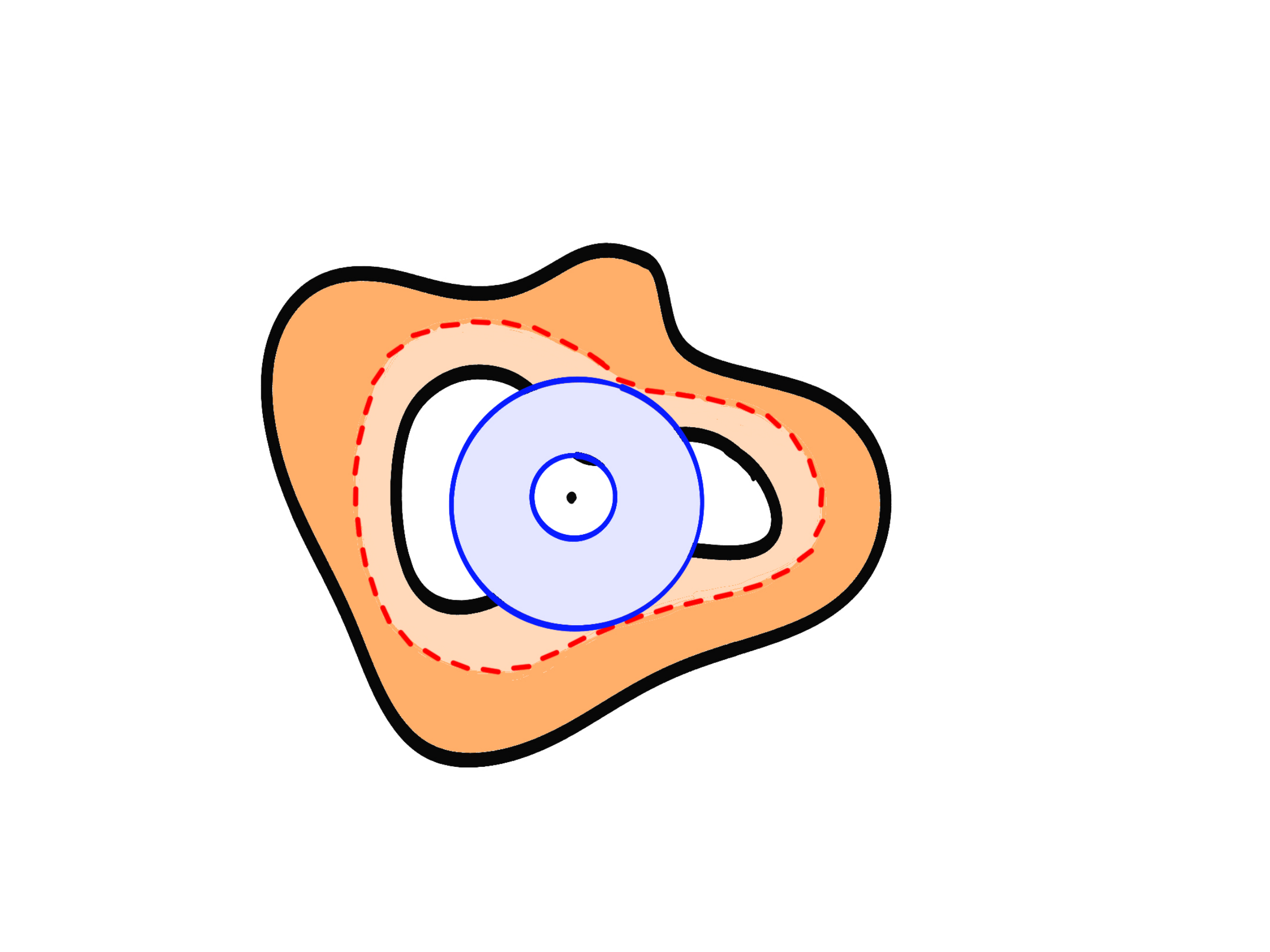}};
    
  \begin{scope}[x={(img.south east)}, y={(img.north west)}]
    \node[anchor=south west, font=\scriptsize] at (0.13,0.55) {$\mathcal{V}_2$};
    \node[ font=\tiny] at (0.455,0.42) {$\mathcal{U}^-_{R_0,M,k,f}$};
    \node[anchor=west, font=\scriptsize] at (0.3,0.6) {$\Gamma$};
    \node[anchor=west, font=\scriptsize] at (0.46,0.70) {$\Omega$};
    \node[anchor=west, font=\scriptsize] at (0.31,0.53) {$\mathcal{V}_1$};
    \node[anchor=west, font=\scriptsize] at (0.41,0.48) {$p$};
    \node[anchor=west, font=\scriptsize] at (0.4,0.28) {$\mathcal{U}$};
  \end{scope}
\end{tikzpicture}
\end{minipage}

\caption{Schematic representation of the two comparison configurations used in the proof of Lemma~\ref{lemma_Adm}. 
The dark orange region represents $\mathcal U$, the dashed red hypersurface represents $\Gamma$, and the blue region denotes the comparison region centered at $p$. 
Left: Case~1, where $\mathcal U\subset \mathcal V_1$ and the family $\mathcal U^+_{R,M,k,f}$ is shrunk from outside $\mathcal V_1$. 
Right: Case~2, where $\mathcal U\subset \mathcal V_2$ and the family $\mathcal U^-_{R,M,k,f}$ is expanded from $\mathcal V_1$.}
\end{figure}

\begin{itemize}
\item[] \textbf{Case 1: \(\mathcal U\subset \mathcal V_{1}\).}
For any point $p \in \mathcal{U}$, let $u_{R,M,k,f}$ denote the model solution centered at $p$, written in terms of $r_p=\textup{dist}(p,\cdot)$. We use the notation $\U_{R,M,k,f}^{\pm}$ from Section~\ref{subsec:ComFram} for the regions where $u_{R,M,k,f}$ is positive and does not attain its maximum.
Choose $R_0 > 0$ so large that $\U_{R_0,M,k,f}^{+} \cap \mathcal{V}_1 = \emptyset$. For $R < R_0$, set
\[
w = u_{R,M,k,f} - u \quad \text{in} \quad \U_{R,M,k,f}^{+} \cap \mathcal{U}.
\]
By decreasing $R$ slightly from $R_0$, we reach a first-contact configuration such that
\begin{enumerate}
	\item $w \leq 0$ in $\U_{R,M,k,f}^{+} \cap \mathcal{U}$,
	\item there exists a point $q$ either in $\U_{R,M,k,f}^{+} \cap \mathcal{U}$ 
	or in $\Gamma_{R,M,k,f}^{+} \cap \partial \mathcal{U}$ such that $w(q) = 0$. 
\end{enumerate}
Since $\overline{\tau}(\U_{R,M,k,f}^{+}) < \overline{\tau}(\mathcal{U})$, the strong maximum principle implies that $u_{R,M,k,f}\equiv u$ in a neighbourhood of $q$, a contradiction.

\item[] \textbf{Case 2:  $\mathcal{U} \subset \mathcal{V}_2$.}
Choose $p \in \mathcal{V}_1$ and $R_0 > 0$ so that $\U_{R_0,M,k,f}^{-} \subset \mathcal{V}_1$. Define
\[
w = u - u_{R,M,k,f} \quad \text{in} \quad \U_{R,M,k,f}^{-} \cap \mathcal{U},
\]
and take $R > R_0$ large. The same first-contact argument yields a contradiction with the strong maximum principle.
\end{itemize}
Both cases are impossible, hence
\[
\overline{\tau}(\mathcal U) \notin \mathrm{Gap}_{M,k,f}.
\]
\end{proof}

\subsection{An isoperimetric inequality}\label{sec:IsoIneq}
In this subsection we prove Theorem~\ref{theo:Isoperimetric} for a component $(\U,u)$ such that $\textup{Max}(u)\cap \textup{cl}(\U)$ contains a Lipschitz-continuous hypersurface. The key input is an area bound for the top level set, extending \cite[Proposition~5.4]{EMa2} to arbitrary dimension and to general level sets.

Throughout, we keep the standing assumptions on $(\M,g)$ and assume that $f\in \mathcal{C}^{m}(\mathbb R)$, where $m=\max\{n-2,3\}$, is $k$-admissible and satisfies \eqref{conditionf}.

\begin{proposition}\label{prop_AreaLevelSet}
 Let $(\bar \U, \bar u )$ be a comparison pair with core radius $\bar R > 0$. Let $t \in [0, M)$ be a regular value of $u$ and set $\Gamma _t := u ^{-1} (t)$. Assume that the $(n-1)$-dimensional part of $\textup{cl}(\U) \cap \textup{Max}(u)$, denoted by $\Gamma_{M} \subset {\rm cl}(\mathcal U) \cap {\rm Max}(u)$, is a Lipschitz-continuous hypersurface. Then, 
	\begin{equation}\label{AreaBound}
		\mathcal{H}^{n-1} (\Gamma_{M}) \leq \left\{ \begin{matrix}
			\left(\dfrac{s_k(\bar{R})}{s_k (\chi_+ (t))}\right)^{n-1} \mathcal{H}^{n-1}(\Gamma _t)  & \text{ if } & (\bar \U, \bar u) \in {\rm Model}_{M,k,f}^+,  \\[6mm]
			\left(\dfrac{s_k(\bar{R})}{s_k (\chi_- (t))}\right)^{n-1} \mathcal{H}^{n-1}(\Gamma _t) &  \text{ if } &  (\bar \U, \bar u) \in {\rm Model}_{M,k,f}^-,
		\end{matrix}\right.
	\end{equation}
	where $\chi_{\pm}$ is given in \eqref{chiFunctions}. Furthermore, equality holds in \eqref{AreaBound} for some $t$ if, and only if $(\U,u) \equiv (\bar \U, \bar u)$.
\end{proposition}

\begin{proof}
The argument follows \cite[Proposition~2.4]{ABBM}, which treats the case $t=0$ for Serrin's equation in $\r^n$. Define the vector field
\begin{equation}\label{X}
\mathcal X (p) :=   \frac{1}{s_k(\Psi)^n \bar{U}' (\Psi)}  \nabla u \quad \text{for all } p \in \Omega  \setminus \textup{Max}(u),
\end{equation}
whose divergence is
\[
\operatorname{div}(\mathcal X ) = \frac{1}{s_k(\Psi) ^{n} \bar{U}' (\Psi)^3} \left( f(u)(W-\bar{W})-\cot_k(\Psi)W \right) .
\]
Since $f$ is of class $\mathcal{C}^m$, it follows that $u \in \mathcal{C}^{m+2} (\U)$, so Sard's Theorem implies that the set of critical values of $u$ in $\U$ has measure zero. Hence, we can fix $\varepsilon>0$ arbitrarily small such that $M-\varepsilon$ is a regular value of $u$, and set
\[
\U_{t, \varepsilon} = \set{p \in \U ~\colon~t < u(p)< M- \varepsilon}.
\]
Applying the divergence theorem in $\U_{t,\varepsilon}$ gives
\begin{equation}\label{integralIdentity}
\begin{aligned}
\int_{\U_{t,\varepsilon}}\frac{1}{s_k (\Psi)^{n} \bar{U}' (\Psi)^3}& \bigg( f(u)(W-\bar{W})-\cot_k(\Psi)W \bigg) \, d\mu =\\ 
& \int_{\Gamma _t} \frac{-\abs{\nabla u}}{s_k(\Psi)^n\bar{U}' (\Psi)} \, d \sigma+\int_{\Gamma_{M-\varepsilon}} \frac{\abs{\nabla u}}{s_k(\Psi)^n\bar{U}' (\Psi)} \, d \sigma,
\end{aligned}
\end{equation}
where $d\mu$ and $d\sigma$ are the volume elements of $\mathcal{M}$ and the induced hypersurface measure, respectively, and the inner unit normal satisfies $\nu = \nabla u / \abs{\nabla u}$ along $\Gamma_t$ and $\nu=-\nabla u/\abs{\nabla u}$ along $\Gamma_{M-\varepsilon}$.

Along $\Gamma_t$ the pseudo-radial function is constant, $\Psi=\chi_\pm(t)$, and Theorem~\ref{theo:Gradient} gives $W\le \bar W$. Therefore,
\begin{equation}\label{1}
\int_{\Gamma _t}  \frac{\abs{\nabla u}}{s_k(\Psi)^n\bar{U}' (\Psi)}\, d\sigma
\left\{	\begin{matrix}
\geq - \dfrac{\mathcal{H}^{n-1}(\Gamma _t)}{s_k^{n-1}(\chi_+(t))} & \text{ if } &(\bar \U, \bar u) \in {\rm Model}_{M,k,f}^+ ,\\[6mm]
\leq \dfrac{\mathcal{H}^{n-1}(\Gamma _t)}{s_k^{n-1}(\chi_-(t))} & \text{ if } &(\bar \U, \bar u) \in {\rm Model}_{M,k,f}^-.
\end{matrix}\right.
\end{equation}

Since $\Gamma_M$ is Lipschitz, it is $(n-1)$-rectifiable and has differentiability points of full $\mathcal H^{n-1}$-measure. Using \cite[Theorem 2]{BanyagaHurtubise04} and the expansions \eqref{eq:expansion1}--\eqref{eq:expansion2}, we have
\[
\lim_{x \in \U,\, x \to p} \frac{W}{\bar W} =1
\quad\text{for }\mathcal H^{n-1}\text{-a.e. }p\in\Gamma_M.
\]
Taking a sequence $\varepsilon_n\to0$ such that $M-\varepsilon_n$ is a regular value for all $n$, we obtain
\begin{equation}\label{2}
\lim_{n \to +\infty} \int_{\Gamma_{M-\varepsilon_n}}  \frac{\abs{\nabla u}}{s_k(\Psi)^n\bar{U}' (\Psi)} \, d \sigma= 
\left\{	\begin{matrix}
- \dfrac{\mathcal{H}^{n-1}(\Gamma_{M})}{s_k^{n-1}(\bar{R})} & \text{ if } & (\bar \U, \bar u) \in {\rm Model}_{M,k,f}^+ ,\\[6mm]
\dfrac{\mathcal{H}^{n-1}(\Gamma_{M})}{s_k^{n-1}(\bar{R})} & \text{ if } & (\bar \U, \bar u) \in {\rm Model}_{M,k,f}^-.
\end{matrix}\right.
\end{equation}

Finally, since $\bar U'(\Psi)^3<0$ on $\U$ if $(\bar \U,\bar u)\in{\rm Model}_{M,k,f}^+$ and $\bar U'(\Psi)^3>0$ if $(\bar \U,\bar u)\in{\rm Model}_{M,k,f}^-$, we conclude that
\begin{equation}\label{3}
\int_{\U}\frac{1}{s_k (\Psi)^{n} \bar{U}' (\Psi)^3}  \bigg( f(u)(W-\bar{W})-\cot_k(\Psi)W \bigg) \, d\mu 
\left\{	\begin{matrix}
\geq 0& \text{ if } & (\bar \U, \bar u) \in {\rm Model}_{M,k,f}^+ ,\\[6mm]
\leq 0 & \text{ if } &(\bar \U, \bar u) \in {\rm Model}_{M,k,f}^-.
\end{matrix}\right.
\end{equation} 

Letting $\varepsilon_n\to0$ in \eqref{integralIdentity} and using \eqref{1}--\eqref{3} yields \eqref{AreaBound}.

For rigidity, note that equality in \eqref{AreaBound} forces equality in \eqref{3}, hence $W=\bar W$ on $\U$, and then $(\U,u)\equiv(\bar \U,\bar u)$ by the rigidity statement in Theorem~\ref{theo:Gradient}.
\end{proof}

Combining Proposition~\ref{prop_AreaLevelSet} with the coarea formula yields Theorem~\ref{theo:Isoperimetric}.

\begin{proof}[Proof of Theorem \ref{theo:Isoperimetric}]
We only consider the case $(\bar \U, \bar u) \in {\rm Model}_{M,k,f}^+$. The case $(\bar \U, \bar u)\in{\rm Model}_{M,k,f}^-$ is identical.

By the Unique Continuation principle for $u$, the set of critical points cannot have interior points; thus the hypotheses of \cite[Proposition 2.2]{cadedduBriefNoteCoarea2018} are satisfied, and the coarea formula can be applied yielding, 
$$
\mathcal{H}^n(\U) = \int_{(0,M)} \left( \int_{\Gamma_t} \frac{1}{|\nabla u|} \, d\sigma_t \right) dt,
$$
where  $d\sigma_t$ denotes the induced area element on $\Gamma_t$. Using Theorem~\ref{theo:Gradient} and \eqref{definW}, we obtain
\begin{equation}\label{eq1}
\mathcal{H}^n (\U) \geq \int _0 ^M \left( \int _{\Gamma _t} \frac{1}{ \sqrt{\bar W} } \, d \sigma _t \right) dt,
\end{equation}
since $\bar W>0$ on $\Gamma_t$ for $t<M$.

Moreover, $\bar W$ is constant along each $\Gamma_t$ and
\[
\bar W = \bar U'(\chi_\pm(t))^2,\quad t\in(0,M),
\]
where we take $\chi_+$ when $\overline{\tau}(\U) \leq \tau_k^+ (M)$ and $\chi_-$ otherwise. Therefore,
\begin{equation}\label{eq2}
\begin{split}
\int _0 ^{M} \left( \int _{\Gamma _t} \frac{1}{ \sqrt{\bar W} } \, d \sigma _t \right) dt
&= \int_0^M \frac{\mathcal H^{n-1}(\Gamma_t)}{|\bar U'(\chi_+(t))|}\,dt \\
&\ge \frac{\mathcal{H}^{n-1}(\Gamma_{M})}{s_k (\bar R)^{n-1}}
\int_{0}^{M} \frac{s_k (\chi_+ (t))^{n-1}}{ |\bar U ' (\chi_+ (t))|} \, dt,
\end{split}
\end{equation}
where we used Proposition~\ref{prop_AreaLevelSet} in the last inequality.

Finally, since $\chi_{+}:[0,M]\to[\bar R,\bar r_+]$ is decreasing, the change of variables $t=\chi_+^{-1}(r)$ yields
\[
\int_{0}^{M} \frac{s_k (\chi_+ (t))^{n-1}}{ |\bar U ' (\chi_+ (t))|} \, dt
= \int_{\bar R}^{\bar{r}_+} s_k (r)^{n-1} \, dr.
\]
Now observe that
\[
\mathcal{H}_k^n ((r_1, r_2) \times \Gamma_0) = \mathcal{H}_k^{n-1} (\Gamma_0) \cdot \int_{r_1}^{r_2} s_k (r)^{n-1} \, dr 
\]
and
\[
\mathcal{H}_k^{n-1} (\{ r_1 \} \times \Gamma_0) = \mathcal{H}_k^{n-1} (\Gamma_0) \cdot s_k (r_1)^{n-1},
\]
for all $r_1, r_2 \in I$ with $r_1 < r_2$.

Hence, combining this with \eqref{eq1}--\eqref{eq2} gives \eqref{volumeBound}. The rigidity statement follows as in Proposition~\ref{prop_AreaLevelSet}
\end{proof}

\subsubsection*{On the extremal case}
We comment on Proposition~\ref{prop_AreaLevelSet} and Theorem~\ref{theo:Isoperimetric} when \(\overline{\tau}(\U)\notin \textup{Adm}_{M,k,f}\).
If \(\textup{Max}(u)\cap \textup{cl}(\U)\) contains a Lipschitz hypersurface, then it has differentiability points of full \((n-1)\)-dimensional measure; in particular, it contains an \((n-1)\)-dimensional regular point. Lemma~\ref{lemma_Adm} then yields
\[
\overline{\tau}(\U) > \tau_{k,f}^0(M),
\]
so \(\overline{\tau}(\U)\notin \textup{Adm}_{M,k,f}\) forces \(\overline{\tau}(\U)\in \mathrm{Gap}_{M,k,f}\).

When \(k>0\), Proposition~\ref{prop:behav} implies \(\mathrm{Gap}_{M,k,f}=\varnothing\), hence the bounds above are always sharp. Assume now \(k\le0\). By Theorem~\ref{theo:Gradient}, the sharpest estimates are obtained by choosing \(\overline{\tau}(\bar\U)\in \textup{Im}(\overline{\tau}_{M,k,f}^-)\) as close as possible to \(\tau_{k,f}^-(M)\). If no \(R\in[0,\infty)\) realizes the infimum,
\[
\overline{\tau}_{M,k,f}^-(R)=\tau_{k,f}^-(M),
\]
then necessarily
\[
\lim_{R\to\infty}\overline{\tau}_{M,k,f}^-(R)=\tau_{k,f}^-(M),
\]
which occurs, for instance, under \eqref{conditionf2}; see Proposition~\ref{prop:ODE2}. In that case, evaluating \eqref{ODERadial} at \(r_-(R,M,k,f)\) gives
\[
0 = U_{R,M,k,f}\bigl(r_-(R,M,k,f)\bigr)
= M - \int_{r_-(R,M,k,f)}^{R} \frac{1}{s_k(y)^{n-1}}
\left( \int_{R}^{y} s_k(x)^{n-1} f\bigl(U_{R,M,k,f}(x)\bigr)\, dx \right) dy,
\]
and hence
\[
\lim_{R \to \infty} r_-(R,M,k,f) = +\infty.
\]
Letting \(R\to\infty\) in \eqref{AreaBound} and \eqref{volumeBound} then yields, in particular,
\[
\mathcal{H}^{n-1}(\Gamma_M) \le \mathcal{H}^{n-1}(\Gamma_t),
\quad \forall\, t \in [0,M],
\]
since
\[
\lim_{R \to \infty} \frac{s_k(R)}{s_k\bigl(r_-(R,M,k,f)\bigr)} = 1.
\]
Indeed, if this limit were different from $1$, then
\[
\lim_{R \to \infty}
\frac{\displaystyle\int_{r_-(R,M,k,f)}^{R} s_k(r)^{n-1}\, dr}
{s_k(R)^{n-1}} = +\infty,
\]
and Theorem~\ref{theo:Isoperimetric} would force \(\U\) to be unbounded, a contradiction. A similar argument also excludes the possibility that the ratio tends to $0$. Thus Theorem~\ref{theo:Isoperimetric} remains non-trivial in this limiting regime, although the limiting inequality is not explicit in general.

\subsection{Location of the hot-spots}\label{sec:hotspots}

In this subsection we estimate the distance from the maximum set of a solution $(\Omega,u)$ of \eqref{DP} to the boundary, extending the sharp bounds in \cite[Section 5]{ABM}.

\begin{proposition}\label{prop_hotSpots}
Let $f\in \mathcal C^{1}(\mathbb R)$ be a $k$-admissible function satisfying \eqref{conditionf2}. Let $p \in \textup{Max}(u) \cap \textup{cl}(\U)$. Then
\begin{equation}\label{ineqHotSpots}
\textup{dist}(p, \partial \Omega \cap \textup{cl}(\U)) \geq \left\{ \begin{matrix}
\bar{r}_+- \bar R  & \text{ if } & (\bar \U, \bar u) \in {\rm Model}_{M,k,f}^+  \\[6mm]
\bar R-\bar{r}_- &  \text{ if } &  (\bar \U, \bar u) \in {\rm Model}_{M,k,f}^-,
\end{matrix}\right.
\end{equation}
where $\bar r_\pm$ are the boundary radii (zeros) of the profile defining $(\bar \U, \bar u)$ in \eqref{ODERadial}--\eqref{CauchyData}. Furthermore, equality in \eqref{ineqHotSpots} holds if, and only if $(\U,u) \equiv (\bar \U, \bar u)$.
\end{proposition}

The proof is identical to \cite[Theorem 5.1]{ABM}, so we omit it. To obtain the normalized estimate in Theorem~\ref{theo:HotSpots}, we also need a pointwise lower bound for $u$ in terms of the distance to $\partial\Omega$, obtained by comparison with the Serrin's ball profile as in \cite[Lemma 2.1]{MP}.

\begin{lemma}\label{lemma_distanceToBoundary}
Let $(\Omega,u)$ solve \eqref{DP} with $f(x) \leq f_k(x)=nkx+1$. Suppose that $\Omega$ is contained in a convex domain of $\mathcal{M}$, and define $d(p) = \textup{dist} (p, \partial \Omega)$ for all $p \in \textup{cl} (\Omega)$. Then
\begin{equation}\label{ineqSerrin}
u(p) \geq \frac{2}{n} \cdot \frac{s_k \left( \frac{d(p)}{2}\right)^2}{s_k ' (d(p))}, \quad \forall p \in \textup{cl} (\Omega).
\end{equation}
Moreover, if equality holds at some point $p \in \Omega$, then $u$ attains its unique maximum at $p$, $\Omega$ is a geodesic ball centered at $p$, $u$ depends only on the distance to $p$, and $f(x)= f_k(x)$.
\end{lemma}

\begin{proof}
Fix \(p \in \Omega\), and write \(r_p(q)=\mathrm{dist}(p,q)\) for \(q\in \textup{cl}(\Omega)\). Under \(\mathrm{Ric} \geq (n-1)k\,g\) one has
\[
\Delta r_p(q) \le (n-1)\cot_k(r_p(q)),\qquad q\in \textup{cl}(\Omega).
\]
Let
\[
W^R(r) = \frac{2\left[ s_k\!\left(\frac{R}{2}\right)^2 - s_k\!\left(\frac{r}{2}\right)^2 \right]}{n\, s_k'(r)}, 
\quad r \in [0, R),
\]
be the Serrin's ball profile of radius \(R\) in \(\mathbb{M}^n(k)\). Since \(f(x)\le nkx+1\), the function $w^R:=W^R\circ r_p$ is a subsolution of \eqref{DP} on the geodesic ball \(B_R(p)\). Taking \(R=d(p)\) and using the maximum principle gives \(u\ge w^{d(p)}\) in \(B_{d(p)}(p)\), hence
\[
u(p)\ge w^{d(p)}(p)=\frac{2}{n}\cdot \frac{s_k\!\left(\frac{d(p)}{2}\right)^2}{s_k'(d(p))}.
\]
The rigidity statement follows from the strong maximum principle.
\end{proof}

We denote by
\[
r_{\Omega} = \max \left\{ d(p)=\textup{dist} (p, \partial \Omega) ~\colon~ p \in \Omega \right\}
\]
the inradius of $\Omega$. We can now prove Theorem~\ref{theo:HotSpots}.

\begin{proof}[Proof of Theorem \ref{theo:HotSpots}]
Let $o \in \Omega$ such that $r_{\Omega}= \textup{dist} (o, \partial \Omega)$. 
Developing the inequality of Lemma~\ref{lemma_distanceToBoundary} for $\frac{2}{n} u(o) + k u (o)^2$ and using the the upper bound for $u$ gives
\[
\frac{\tan_k (r_\Omega)^2}{n^2} \leq \frac{2}{n} u(o) + k u (o)^2 \leq \frac{2}{n}M+ kM^2.
\]
Combining this bound with Proposition~\ref{prop_hotSpots} produces the lower bound for the function $ \frac{ n d(p)}{\tan_k(r_\Omega)}$ that yields the equation \eqref{ineqHotSpotsNormalized} .

For rigidity, equality in \eqref{ineqHotSpotsNormalized} implies equality in \eqref{ineqSerrin}, hence $f(x)=nkx+1$, $\Omega$ is a geodesic ball centered at $o$, $u$ depends only on $r_o$, and $\textup{Max}(u)=\{o\}$. In particular, we get that equality can hold only if $(\bar \U, \bar u) \in \textup{Model}_{M,k,f}^+$. Since $\Omega\setminus \textup{Max}(u)$ is connected, equality in \eqref{ineqHotSpots} also holds and the rigidity statement in Theorem~\ref{theo:Gradient} yields the warped product decomposition
\[
g = d\Psi^2+ s_{k} (\Psi)^2 \bar g
\]
on $\Omega$, where $\Psi$ is as in Definition~\ref{def:pseudo-radial} and $\bar g$ is the metric induced on $\Gamma_0$. Since then $\tau(\Omega)=1$, the conclusion follows from Remark~\ref{remark_Gradient}.
\end{proof}

\section{Exotic solutions on  \texorpdfstring{$\mathbb{S}^n$}{Sphere} and on quotients}\label{sec:ExoticSol}

Let $\Omega\subset \mathbb{S}^n$ be a domain with $\mathcal{C}^2$-boundary, and let $f\in \mathcal{C}^1(\mathbb{R})$. We consider the overdetermined elliptic problem (OEP),
\begin{equation}\label{OEP}
\left\{
\begin{array}{rll}
\Delta u + f(u) = 0 & \text{in} & \Omega,\\
u>0 & \text{in} & \Omega,\\
u=0 & \text{on} & \partial\Omega,\\
| \nabla u| = \alpha_i & \text{on} & \Gamma_i \subset \partial\Omega,
\end{array}
\right.
\end{equation}
where $\Delta$ and $\nabla$ denote the Laplace--Beltrami operator and gradient with respect to $g_{\mathbb{S}^n}$, and $\alpha_i$ is constant on each connected component $\Gamma_i$ of $\partial\Omega$.

Unlike space forms of non-positive curvature, the sphere $\s^n$ admits rich isoparametric foliations. In our construction, the level sets of these solutions are the leaves of an isoparametric foliation (see, e.g., \cite{DV}) this allows us to construct exotic solutions of the OEP. Moreover, we show that, whenever a group of isometries preserves the foliation, the construction descends to domains in a smooth quotient $\s^n/G$.

\subsection{Isoparametric Hypersurfaces of \texorpdfstring{$\mathbb{S}^n$}{Sphere}}\label{appendixISO}

Isoparametric hypersurfaces have been extensively studied and provide
natural foliations of the sphere by level sets of a homogeneous polynomial.

A fundamental result in the theory is due to M\"uzner: given an isoparametric
family on $\mathbb S^n$, there exists a homogeneous polynomial
$P:\mathbb R^{n+1}\to\mathbb R$ of degree $\ell\in\{1,2,3,4,6\}$ (the
\emph{Cartan--M\"unzner polynomial}) whose restriction to $\mathbb S^n$
has the leaves of the family as its level sets; see \cite{CR,KFM}.
Moreover, $P$ satisfies
\begin{equation}\label{eq:Cartan}
  \bar{\Delta} P (x) = c \abs{x}^{\ell-2}, \qquad
  |\bar{\nabla} P (x)| = \ell^2 |x|^{2\ell-2},
\end{equation}
where $c=\tfrac{P^2}{2}(m_2-m_1)$, $(m_1,m_2)$ are the principal-curvature
multiplicities, and $\bar{\Delta}$, $\bar{\nabla}$ are the Euclidean Laplacian
and gradient.

\begin{theorem}[Isoparametric hypersurfaces in $\mathbb{S}^n$]\label{theorIsoparametric}
If $M^{n-1}\subset \mathbb{S}^n$ is an isoparametric hypersurface with $l$ distinct principal curvatures and multiplicities $m_1,m_2$, then there exists a homogeneous polynomial $P:\mathbb{R}^{n+1}\to\mathbb{R}$ of degree $\ell$ (the Cartan--M\"unzner polynomial) such that $M=P^{-1}(0)\cap\mathbb{S}^n$. Moreover, $P$ satisfies \eqref{eq:Cartan}.

Conversely, if $P$ is a homogeneous polynomial of degree $\ell$ satisfying \eqref{eq:Cartan}, then each regular level set of $P|_{\mathbb S^n}$ is an isoparametric hypersurface with $\ell$ distinct principal curvatures.
\end{theorem}

We recall two facts that are repeatedly used in the isoparametric reduction. First, the level-set function $s=\arccos(P|_{\mathbb S^n})/\ell$ is smooth on the complement of the focal varieties and satisfies
\begin{equation}\label{relationsV}
  |\nabla s|=1, \qquad
  \Delta s=(n-1)\cot(\ell s)-\frac{c}{\ell\sin(\ell s)},
\end{equation}
where $\nabla,\Delta$ are computed with respect to the round metric of $\mathbb S^n$. We can list some of the main examples of isoparametric hypersurfaces as: 

\noindent\textbf{(A) The case $\ell=1$.}
The only possibility is the totally geodesic sphere $\mathbb S^{n-1}\subset \mathbb S^n$.

\smallskip
\noindent\textbf{(B) The case $\ell=2$.}
Up to congruence, the hypersurfaces are the products of spheres $\mathbb S^{m_1}\bigl(\sqrt{(1+c)/2}\bigr)\times \mathbb S^{m_2}\bigl(\sqrt{(1-c)/2}\bigr)\subset\mathbb S^n$, with $(m_1,m_2)=(p,q)$ and $p+q=n-1$. One can take
\[
P(x,y)=|x|^2-|y|^2,\qquad (x,y)\in\mathbb R^{p+1}\times\mathbb R^{q+1},
\]
so that $c=m_2-m_1$.

\smallskip
\noindent\textbf{(C) The case $\ell=3$.}
These are the Cartan hypersurfaces, with equal multiplicities
$m_1=m_2=m\in\{1,2,4,8\}$, arising from the isoparametric representations
associated with the real, complex, quaternionic and octonionic projective planes;
see \cite{CR}.

\smallskip
\noindent\textbf{(D) The case $\ell=4$.}
There are homogeneous families (classified by Hsiang--Lawson) and also
inhomogeneous ones; the complete classification is known. For the purposes of
this paper, we only use that $\ell=4$ families exist for many pairs $(m_1,m_2)$;
see \cite{CR,KFM}.

\smallskip
\noindent\textbf{(E) The case $\ell=6$.}
All known examples are homogeneous and are conjectured to exhaust all cases.
Partial results include \cite{Nomizu71,Miyaoka93} (see also \cite{KFM,Wang}).

\subsection{Construction along an isoparametric foliation of  \texorpdfstring{$\mathbb{S}^n$}{Sphere}}

Let $\{\Gamma_r\}_{r\in[-1,1]}$ be an isoparametric family in $\s^n$, given by the restriction of a Cartan--M\"unzner polynomial $\rho\in\mathcal C^\omega(\s^n)$ via $\Gamma_r:=\rho^{-1}(r)$. Let $\ell\in\{1,2,3,4,6\}$ be the number of distinct principal curvatures and let $m_1,m_2$ denote their multiplicities (alternating when $\ell$ is even, coincident when $\ell$ is odd). Along a normal geodesic to a focal submanifold, the Cartan--M\"unzner theory (see e.g.\ \cite[\textsection 4.1]{DV}) yields
\begin{equation}\label{eq:rho-cos}
  \rho(p)=\cos\!\big(\ell\, s(p)\big),\qquad s(p):={\rm dist}_{\sp^n}(p,\Gamma_1).
\end{equation}
Let $f\in \mathcal{C}(\r)$ be a continuous real function. Seeking $z=Z\circ s$, a direct computation shows that $z$ solves the PDE in 
\eqref{DP} if and only if $Z$ solves
\begin{equation}\label{eq:ODE-isop}
  Z''(s) + \left( (n-1)\cot(\ell s) - \frac{c}{\ell\sin(\ell s)} \right)Z'(s) + f\big(Z(s)\big)=0,
\end{equation}
where $c=\tfrac{\ell^2}{2}(m_2-m_1)$. We prescribe the Cauchy data
\begin{equation}\label{eq:Cauchy-isop}
  Z(S)=M>0,\qquad Z'(S)=0,\qquad S\in[0,\pi/\ell].
\end{equation}
We want a solution to \eqref{eq:ODE-isop}--\eqref{eq:Cauchy-isop} to produce a pair $(\Omega,u)$ solving the overdetermined problem \eqref{OEP} in $\s^n$. This leads us to the following definition, extending the notion of $k$-admissibility:

\begin{definition}[Admissible nonlinearities]\label{def:admissibleIsoparametric-f}
A locally Lipschitz function $f\in \textup{Lip}_{loc}(\mathbb{R})$ is called \emph{Iso-admissible} if there exist non-empty connected intervals 
	\[
	0 \in  \mathcal{S}_f \subset [0,\pi/\ell]
	\quad\text{and}\quad
 \mathcal{I}_f \subset \r_+ \text{ open with } 0 \in \partial  \mathcal{I}_f, 
	\]
	such that, for every \((S,M)\in \mathcal S_f\times \mathcal I_f\), the unique maximal solution \(Z_{S,M,f}\) of \eqref{eq:ODE-isop}--\eqref{eq:Cauchy-isop} is defined
	on an interval that contains a closed interval \([s_-(S,M,f),s_+(S,M,f)]\) with
	$$
	0\le s_-(S,M,f)\leq S<s_+(S,M,f)\le \pi/\ell,
	$$
	and satisfies:
	\begin{itemize}
		\item If $S=0$, then $s_-(0,M,f) = 0 < s_+(0,M,f)$, 
		\begin{itemize}
			\item $Z_{0,M,f} > 0$  on $ [0,s_+(0,M,f))$, 
			\item $Z'_{0,M,f} \neq 0 $  on $ (0,s_+(0,M,f))$, and 
			\item $Z_{0,M,f}(s_+(0,M,f)) = 0$.
		\end{itemize}
		
		\item If $S>0$, then $ 0 < s_-(S,M,f) < S < s_+(S,M,f) $, 
		\begin{itemize}
			\item $Z_{S,M,f} > 0$  on $(s_-(S,M,f),s_+(S,M,f))$, 
			\item $Z'_{S,M,f} \neq 0$  on $(s_-(S,M,f),s_+(S,M,f)) \setminus \{S \}$, and 
			\item $Z_{S,M,f}(s_\pm(S,M,f)) = 0$.
		\end{itemize}
	\end{itemize}
\end{definition}

Standard arguments, presented in Appendix~\ref{app:isoparametric-ode}, on ODE and focal geometry allow us to show:

\begin{lemma}\label{lemma_exist-ODE-isop} 
Let $f \in \textup{Lip}_{loc} (\r)$ and $\mathcal{I}_f \subset \r_+$ an open interval with $0 \in \partial \mathcal{I}_f$ such that $f>0$ on $\mathcal{I}_f$. Then $f$ is Iso-admissible with $\mathcal{S}_f = [0, \pi/\ell]$.
\end{lemma}

Then, writing $R= \cos(\ell S)\in[-1,1]$ for $S \in [0,\pi/\ell]$, the function $z_{R,M,f}:=Z_{S,M,f}\circ s$ defined on 
\[
\Omega_{R,M,f}:= \left\{ \begin{matrix} 
\{\,p\in\s^n:\  s(p) < s_+(0,M,f)\,\} & \text{ if } R = 1 ,\\[3mm]
\{\,p\in\s^n:\  s(p) > s_-(\pi/\ell,M,f)\,\} & \text{ if } R = -1 ,\\[3mm]
\{\,p\in\s^n:\ s_-(S,M,f)< s(p) < s_+(S,M,f)\,\} & \text{ if } R \in (-1 ,1),
\end{matrix}\right.
\]
solves \eqref{OEP} in $\Omega_{R,M,f}$ with $(z_{R,M,f})_{\max}=M$ and ${\rm Max}(z_{R,M,f})=\Gamma_R$.

\begin{theorem}[Exotic solutions on $\s^n$]\label{cor:exotic-sphere}
Let $f\in \operatorname{Lip}_{loc}(\r_+^*)$ be a positive function on $\r_+^*$. For every isoparametric family 
$\{\Gamma_r\}\subset\s^n$, every $R\in[-1,1]$ and every $M>0$, there exists an $f$-extremal domain
 $(\Omega_{R,M,f},z_{R,M,f})$ for \eqref{OEP} whose level sets are the leaves of $\{\Gamma_r\}$. 
 In particular, when the leaves are not umbilic (e.g.\ Cartan’s families with $\ell=3$), 
 these are non-rotational (exotic) solutions.
\end{theorem}

\begin{remark}\label{rk:Shk-Savo}
For $f(u)=\lambda u$ or $f\equiv1$, the existence of extremal domains associated to isoparametric foliations 
(including cases where the top level set is a focal submanifold) is classical;
see Shklover \cite{Shk} and Savo \cite{Savo}. Our construction yields the same phenomenon for general positive $f$.
\end{remark}

\subsection{Descent to quotients}\label{subsec:descend}

Let $G\le\mathrm{Iso}(\s^n)$ be a closed subgroup and $\{\Gamma_r\}$ an isoparametric family. We say that $G$ \emph{preserves} the foliation if $\rho\circ\xi=\rho$ for all $\xi\in G$. Fix an interval of regular leaves $[a,b]\subset(-1,1)$ on which $G$ acts freely. Then the quotient map $\pi:\s^n\to\s^n/G$ sends each leaf $\Gamma_r$ to a smooth embedded hypersurface $\pi(\Gamma_r)$ in $\s^n/G$, and 
\[
\Omega^\Gamma_{[a,b]}:=\bigcup_{r\in[a,b]}\Gamma_r/G
\]
is a smooth domain in $\s^n/G$ with boundary $\pi(\Gamma_a)\cup\pi(\Gamma_b)$.

\begin{theorem}[Exotic solutions on smooth quotients]\label{thm:descend-quotient}
Assume $G$ preserves $\{\Gamma_r\}$ and acts freely on $\Gamma_r$ for all $r\in[a,b]$. Let $(\Omega_{R,M,f},z_{R,M,f})$ be an $f$-extremal domain as above with $\max(z_{R,M,f})=\Gamma_R$ for some $R\in[a,b]$. Then $z_{R,M,f}$ is $G$-invariant and descends to a function $u_G$ on $\Omega^\Gamma_{[a,b]}$ that solves \eqref{OEP} in the quotient, with the same Dirichlet data and a constant Neumann datum on each connected component of the boundary.
\end{theorem}
\begin{proof}
Since $z_{R,M,f}(p)=Z(s(p))$ depends only on $s$ and $\rho=\cos(\ell s)$ is $G$-invariant, it follows that $z_{R,M,f}$ is invariant under $G$. Hence $z_{R,M,f}$ is constant along the $G$-orbits and defines a smooth function $u_G$ on the smooth quotient domain $\Omega^\Gamma_{[a,b]}$. The boundary components of $\Omega^\Gamma_{[a,b]}$ are the projected leaves $\pi(\Gamma_a)$ and $\pi(\Gamma_b)$, and $u_G$ depends only on the leaf parameter; in particular, its normal derivative is constant on each boundary component, as claimed.
\end{proof}

\begin{remark}[Lens spaces and projective spaces]\label{ex:lens-projective}
(1) If $n=2m+1$ and the foliation is invariant under the Hopf $\mathbb{S}^1$-action on $\s^{2m+1}$, then it descends to an isoparametric foliation on $\mathbb{CP}^m$ (and similarly from $\s^{4m+3}$ to $\mathbb{HP}^m$) \cite[Prop.\ 8.1]{DV}. Hence for any interval $[a,b]$ of regular leaves in $\s^{2m+1}$, the corresponding exotic solutions descend to $\mathbb{CP}^m$ and further to any lens space $L(p,q)=\s^{2m+1}/\Z_p$ (since the $\Z_p$-action is free).

(2) For the antipodal quotient $\mathbb{RP}^n=\s^n/\{\pm1\}$, an isoparametric family descends precisely when $\ell$ is even (so that $\rho$ is $\{\pm1\}$-invariant). When $\ell$ is odd, the foliation does not descend under the antipodal map.
\end{remark}

\begin{remark}[Immersed case and singular quotients]\label{rk:immersed}
If the $G$-action has fixed points on some leaf $\Gamma_r$, then $\s^n/G$ is an orbifold near the fixed set and $\pi(\Gamma_r)$ is an immersed hypersurface with lower-dimensional singular strata. The function $z_{R,M,f}$ still descends to the regular part of the quotient and solves \eqref{OEP} there, with constant Neumann data on each connected piece of $\partial\Omega^\Gamma_{[a,b]}$ away from the singular set. For the purposes of a smooth boundary problem, we restrict to intervals where the $G$-action is free on each leaf, so that $\Omega^\Gamma_{[a,b]}$ is a smooth manifold with boundary.
\end{remark}

\appendix
\AtAppendix{\counterwithin{lemma}{section}}

\section{Appendix: ODE toolkit}\label{appendixSolutions}

In this appendix we collect and prove the ODE statements used in Sections~\ref{sec:ComMethRes} and \ref{sec:ExoticSol}. We first treat the general (isoparametric) reduction on $\s^n$ and then specialize to the equation related to $f$-extremal domains in model manifolds. Since most of the arguments are standard and already considered in \cite{EMa2}, we only sketch the proofs and refer to our previous paper when necessary.

\subsection{Solutions related to the isoparametric case}\label{app:isoparametric-ode}

Fix a locally Lipschitz nonlinearity $f\in \textup{Lip}_{loc}(\r)$, positive on an open interval $\mathcal{I}_f \subset \r_+$. We study the ODE
\begin{equation}\label{eq:V-ODE}
	Z''(s)+\Big((n-1)\cot(\ell s)- \frac{c}{\ell\sin(\ell s)}\Big)\,Z'(s)+ f\big(Z(s)\big)=0\quad\text{on }(0,\pi/\ell),
\end{equation}
where $c=\ell^2(m_2-m_1)/2$ and $m_1,m_2$ are the principal-curvature multiplicities of the family.

We analyze the Cauchy problem with interior or endpoint critical data
\begin{equation}\label{eq:Cauchy}
	Z(S)=M>0,\qquad Z'(S)=0,\qquad S\in[0,\pi/\ell].
\end{equation}
First we deal with the nonsingular case, i.e., when $S \in (0, \pi/\ell)$. 

\begin{proof}[Proof of Lemma \ref{lemma_exist-ODE-isop}]
	Existence and uniqueness of a solution $Z_{S,M,f}$ to \eqref{eq:V-ODE}--\eqref{eq:Cauchy} on an open interval containing $S$ follow from the classical theory of ODEs with locally Lipschitz coefficients, so it remains to prove the monotonicity properties of the solution.
	
	Writing the equation in divergence form as $(\mu Z')'+\mu f(Z)=0$ with the smooth integrating factor
	\[
	\mu(s)=\exp \left(\int_S^s \Big((n-1)\cot(\ell x)-\frac{c}{\ell\sin(\ell x)}\Big)\,dx \right),
	\]
	we obtain the identity
	\[
	Z_{S,M,f}'(s)=-\,\mu(s)^{-1}\int_S^s \mu(x)\,f\big(Z_{S,M,f}(x)\big)\,dx,
	\qquad s\in(e_1,e_2),
	\]
	where $(e_1,e_2) \subset (0, \pi/\ell)$ is the maximal interval of definition of $Z_{S,M,f}$. 
	Since $f\big(Z_{S,M,f}\big)>0$ as long as $Z_{S,M,f}>0$, it follows that $Z_{S,M,f}'(s)>0$ for $s<S$ and $Z_{S,M,f}'(s)<0$ for $s>S$, as long as $Z_{S,M,f}$ remains positive. The existence of $s_{-} (S,M,f)$ and $s_+ (S,M,f)$ then follows by the same argument used in the proof of item (2) in \cite[Theorem 4.1]{EMa2}.

Now we deal with the existence in the singular case, i.e., when $S\in \{0, \pi/\ell\}$. In this case, we proceed with a standard fixed point argument, as in \cite{EM,EMa2}.
We only treat the case $S=0$; the case $S=\pi/\ell$ is completely analogous. As in \cite[Subsection 3.2.1]{EM}, for a non-negative integer \( k \) and \( \varepsilon \in (0,\pi) \), we denote by \( \mathcal{C}^k_e (\varepsilon) \) the space of even \( \mathcal{C}^k \)-functions on \( [-\varepsilon, \varepsilon] \), equipped with the norm
	\[
	\norm{Z}_{k,\varepsilon} = \sum_{i=1}^k \sup_{s \in [-\varepsilon, \varepsilon]} \abs{Z^{(i)} (s)}, \quad \forall Z \in \mathcal{C}^k_e (\varepsilon).
	\]
	We also define \( \mathcal{C}_e (\varepsilon) := \mathcal{C}^0_e (\varepsilon) \). For any $g \in \mathcal{C}_e (\varepsilon)$, define
	\begin{equation}\label{definitionA}
		A(g)(s) =-\int_0^s \mu(y)^{-1}\Big(\int_0^y \mu(x)\,g(x)\,dx\Big)\,dy, \quad \forall s \in (-\varepsilon,\varepsilon),
	\end{equation}
	and set $A(g)(0)=0$. It is straightforward to check that $A(g)$ is $\mathcal{C}^2$ on $(-\varepsilon,\varepsilon)$.

	Define the operator \( \mathcal{T} : \mathcal{C}^1_e (\varepsilon) \to  \mathcal{C}^1_e (\varepsilon) \) by
	\[
	\mathcal{T}(Z) = A\big(f(M+Z)\big).
	\]
	Taking into account that
	\[
	\cot (\ell s) = \frac{1}{\ell s} + O(s) \quad \textup{and} \quad \sin (\ell s) = \ell s+ O(s^3)
	\]
	near $0$, we check that there exist constants $C_1,C_2>0$ such that
	\[
	C_1s^{\ell m_1} \leq  \mu (s) \leq C_2 s^{\ell m_1}
	\]
	for $s$ close to $0$. In particular, there exists $C_3>0$ such that
	\[
	\norm{\mathcal{T}(Z)}_{0,\varepsilon} \leq C_3 \,\varepsilon^2 \, \sup_{[M-\norm{Z}_{0,\varepsilon},\,M+\norm{Z}_{0,\varepsilon}]} \abs{f}.
	\]
	Thus, if $\varepsilon$ is small enough, we obtain \( \mathcal{T}(B) \subsetneq B \), where \( B \subset \mathcal{C}_e (\varepsilon) \) is the closed ball of radius $1$ centered at $0$. Consequently, by Schauder's fixed point theorem, \( \mathcal{T} \) has a fixed point \( Z_{0,M,f} \in B \), which is a weak solution to \eqref{eq:V-ODE}--\eqref{eq:Cauchy}. By construction, $Z_{0,M,f}$ is $\mathcal{C}^2$ on $(-\varepsilon,\varepsilon)$, hence it is a classical solution there. The existence of $s_+(0,M,f)$ and the stated qualitative properties follow as in the nonsingular case.
\end{proof}

\subsection{Solutions related to model manifolds}\label{app:radial-ode}

Assume now that $f \in \mathcal{C}(\r)$ satisfies the \nameref{quote:StandardCondition}. We study the equation
\begin{equation}\label{eq:U-ODE}
 U''(r)+(n-1)\cot_k(r)\,U'(r)+f\big(U(r)\big)=0,\qquad r\in(0, \bar r_k),
\end{equation}
together with the initial conditions
\begin{equation}\label{eq:Cauchy2}
	U(R)=M>0,\qquad U'(R)=0,\qquad R\in[0,2 \bar r_k),
\end{equation}
with $\bar r_k=+\infty$ if $k\le0$ and $\bar r_k=\pi/\sqrt{k}$ if $k>0$.

We begin by checking the monotonicity properties of $U_{R,M,k,f}$. In particular, we prove that $f$ is $k$-admissible under the standard assumptions.

\begin{proof}[Proof of Proposition \ref{prop_ODE1}]
	When $k>0$, this is a particular case of the isoparametric setting treated in Lemma \ref{lemma_exist-ODE-isop}, so we assume that $k \leq 0$. Since the argument is the same in both cases, we only present the proof when $k<0$.
	
	Given any $R \in [0,+\infty)$ and $M \in \mathcal{I}_f$, the existence of a solution $U_{R,M,k,f}$ to \eqref{eq:U-ODE}--\eqref{eq:Cauchy2} follows by standard ODE theory, so we omit the details. Let $(e_1, e_2)$ be the maximal interval of definition of $U_{R,M,k,f}$. From the implicit representation
	\[
	U_{R,M,k,f}(r) = M - \int_{R}^r \frac{1}{s_{k} (y)^{n-1}} \left(\int_{R}^y s_{k} (x)^{n-1} f(U_{R,M,k,f} (x)) \, dx \right) \, dy,
	\]
	we obtain that, whenever $U_{R,M,k,f}(r)$ is positive, one has $U_{R,M,k,f}'(r) >0$ for $r<R$ and $U_{R,M,k,f}'(r) <0$ for $r>R$.
	
	The existence of $r_- (R,M,k,f)$ and $r_+(R,M,k,f)$ when $e_2 < +\infty$ is immediate from the boundary behavior of solutions to differential equations with Lipschitz coefficients. Thus, we only prove the existence of $r_+ (R,M,k,f) > R$ when $e_2 = +\infty$.

	Since $f$ satisfies the \nameref{quote:StandardCondition}, suppose that $f(x) \geq -nx+1$. 
	Given a fixed point $p \in \mathbb{H}^n (k)$, define the function 
	$r_p= \textup{dist}_{\mathbb{H}^n (k)} (p,\cdot)$. Then set 
	$\tilde{u} = U_{R,M,k,f} \circ r_p$ inside the set $\U :=\{ r_p (q) > R \}$. 
	Note that $\tilde{u}$ would be a piece of a model solution in $\mathbb{H}^n (k)$ 
	if there exist $r_- (R,M,k,f) \leq R < r_+ (R,M,k,f)$ as in Definition \ref{def:admissible-f}. 
	Take $q \in \set{r_p = R}$. There exists a function $v$ and a geodesic ball 
	$B \subset \mathbb{H}^{n} (k)$ centered at $q$ 
	such that $v$ is positive and solves $\Delta v -nv+1 = 0$ in $B$, 
	and satisfies $v=0$ on $\partial B$. 
	
	Arguing by contradiction, suppose that $\tilde{u}$ is positive in all of $\U$. 
	Then there exists $\alpha_0 >0$ such that $\alpha_0 v < \tilde{u}$ in $B$. 
	By increasing $\alpha$, we obtain $\alpha_1 > \alpha_0$ such that the graphs of $\alpha_1 v$ and $\tilde{u}$ touch tangentially for the first time.
	But then, since $f(x) \geq -nx+1$, the maximum principle implies $\tilde{u} \equiv \alpha_1 v$ in $B$, which is a contradiction.

 \end{proof}

\begin{proposition}\label{prop:infitytau}
	Let $f(x)=f_k(x)=nkx+1$. Then for the corresponding solution of \eqref{ODERadial}--\eqref{CauchyData}, one has
	\[
	\lim_{R\to0^+} \overline\tau_{M,k,f}^-(R) = +\infty.
	\]
\end{proposition}

\begin{proof}
	Denote by $V_{R,M,k}$ the solution to \eqref{ODERadial}--\eqref{CauchyData} with $f=f_k$. For the case $k=0$ the behavior of $\overline\tau_{M,k,f}^-(R)$ is proved in \cite{ABM, ABBM}. Thus, we focus on the case $k \neq 0$. We shall assume that $k = \pm 1$, since otherwise we can work with $\tilde{V}_{R,M,k} (r) = \abs{k} V_{R,M,k} (r/\sqrt{\abs{k}})$. In this case, the function $V_{R,M,k}$ can be expressed as
	\begin{equation*}
		V_{R,M,k}(r) = M + c_k(r)\left( A(R,M,k)+ B(R,M,k)\left( \frac{-k}{s_k(r)} + G_k(r)\right)\right),
	\end{equation*}
	for some positive numbers $A(R,M,k)$ and $B(R,M,k)$ and
	\begin{equation*}
		G_k(r) := 
		\left\{
		\begin{array}{ll}
			\displaystyle\int_0^{\cosh(r)} \frac{1-(1-t^2)^{\frac{n}{2}}}{t^2(1-t^2)^{\frac{n}{2}}}\,dt, & k=-1,\\[3mm]
			\displaystyle\int_0^{\cos(r)} \frac{1-(t^2-1)^{\frac{n}{2}}}{t^2(t^2-1)^{\frac{n}{2}}}\,dt, & k=1.
		\end{array}
		\right.
	\end{equation*}
	Since $G_k(r) \approx r^{1-n}$ near $r=0$, we deduce that $V_{R,M,k}'(r_\pm)$ diverges as $R\to0^+$, and therefore $\overline\tau_{M,k,f}^-(R)\to+\infty$.
\end{proof}

\begin{proof}[Proof of Proposition \ref{prop:behav}]
First we assume $f(0)>0$ also when $k>0$, and by scaling we may fix $f(0)=1$. Suppose for contradiction that 
	\[
	\lim_{R\to 0^+} U'_{R,M,k}(r_-(R,M,k))=\alpha<+\infty .
	\]
	Let $V_{R,M,k}$ be the Serrin's solution for $f_k(x)=nkx+1$. By Proposition \ref{prop:infitytau}, for sufficiently small $R$ we have $V'_{R,M,k}(r_-(R,M,k))>\alpha$. Since $f(x)\ge f_k(x)$, a comparison of $U_{R,M,k,f}$ and $V_{R,M,k}$ on their common domain (via the maximum principle) yields a contradiction. Hence $\overline\tau_{M,k,f}^-(R)\to+\infty$ as $R\to0^+$.  
	
	For the second claim, if $k>0$ then the symmetry relation in Proposition \ref{prop_ODE1} gives $\overline\tau_{M,k,f}^+(\bar r_k)=\overline\tau_{M,k,f}^-(\bar r_k)$, so $ \tau_{k,f}^+(M)\ge  \tau_{k,f}^-(M)$. If $k\le0$, observe that $\overline\tau_{M,k,f}^+(0)=1$ while $\overline\tau_{M,k,f}^-(0^+)=+\infty$, so $\overline\tau_{M,k,f}^+(R)<\overline\tau_{M,k,f}^-(R)$ for small $R>0$. If there were some $R>0$ with $\overline\tau_{M,k,f}^+(R)=\overline\tau_{M,k,f}^-(R)$, then the corresponding model solution would have constant Neumann data on the boundary, and thus (by Serrin's theorem \cite{Se,KP}) $\Omega_{R,M,k,f}$ would be a geodesic ball. But this can only occur if $R=0$, a contradiction. Therefore $\overline\tau_{M,k,f}^+<\overline\tau_{M,k,f}^-$ for all $R>0$, yielding $ \tau_{k,f}^+ (M) \le  \tau_{k,f}^- (M)$. 
	
	In the case $f(0)=0$ when $k>0$, the proposition follows with the same arguments using the function $V_{R,M,k} + \frac{1}{nk}$.
\end{proof}

Now suppose further that $f \in \mathcal{C}^1 (\r)$ and there exists $\bar M>0$ such that
\[
f(x) \geq f'(x) x \quad \forall x \in (0,\bar M) \subset \mathcal{I}_f.
\]
In this case, we can prove the monotonicity properties of the functions $\overline{\tau}_{M,k,f}^{\pm}$ defined in \eqref{eq:tau-model}, as stated in Proposition \ref{prop:ODE2}. The key observation is that the function $H_{R,k,f}:= \partial_R (U_{R,M,k,f})$ satisfies the differential equation
\[
 H''(r)+(n-1)\cot_k(r)\,H'(r)+f' \big(U(r)\big) H (r)=0, \quad \forall r \in [\bar r_-, \bar r_+],
\]
where we write $\bar r_{\pm} = r_{\pm} (R,M,k,f)$. Hence, we can follow the proof of item (4) in \cite[Theorem 4.1]{EMa2}.

\begin{proof}[Proof of Proposition \ref{prop:ODE2}]
	To obtain the monotonicity properties of the functions $\overline\tau_{M,k,f}^{\pm}$, it is enough to study the sign of the functions
	\[
	\varphi_{M,k,f}^{\pm}(R) := \partial_R \big(U_{R,M,k,f} (\bar r_{\pm})^2\big) = -2 \left( U_{R,M,k,f}'' (\bar r_\pm) H_{R,k,f} ( \bar r_\pm)-H_{R,k,f} ' (\bar r_\pm) U_{R,M,k,f} ' (\bar r_\pm) \right).
	\]
	Proceeding as in the proofs of Claims 1 and 2 in \cite[Theorem 4.1]{EMa2}, using the Killing vector field
	\[
	  \Tilde{Y}  = \cos \theta_{n-1} \prod_{j=1}^{n-2} \sin \theta_j  \partial_r+ \cot_k (r) \left( \cos \theta_{n-1} \displaystyle\sum_{i=1}^{n-2} \frac{\prod_{j=1}^{n-2} \sin \theta_j}{\prod_{j=1}^{i-1} \sin \theta_j}  \partial_{\theta_i} - \frac{\sin \theta_{n-1}}{\prod_{j=1}^{n-2} \sin \theta_j} \partial_{\theta_{n-1}} \right)
	\]
	in $\mathbb{M}^n (k)$, we obtain that $\varphi_{M,k,f}^{-}$ and $\varphi_{M,k,f}^{+}$ are negative and positive, respectively, on $[0, \bar r_k)$. In particular, $\overline\tau_{M,k,f}^{-}$ and $\overline\tau_{M,k,f}^{+}$ are strictly decreasing and strictly increasing, respectively. Note that $\overline\tau_{M,k,f}^{+} (0) = 1$ by definition, and it is clear that
	\[
	\lim_{R \to 0^+} \overline\tau_{M,k,f}^{-}(R) = +\infty.
	\]
	On the other hand, one has
	\[
	\lim_{R \to \bar{r}_k} \overline\tau_{M,k,f}^{\pm} (R) = \tau_{k,f}^\pm (M) >0,
	\]
	and the implicit function theorem implies that these quantities depend $\mathcal{C}^1$ on $M$.
	
	In the case $k>0$, the symmetry of the functions $U_{R, M,k,f}$ with respect to $\bar r_k/2$ implies that $\overline\tau_{M,k,f}^{-} (\bar r_k)=\overline\tau_{M,k,f}^{+} (\bar r_k)$.
\end{proof}

\section{Appendix: Gradient Estimates}\label{appendixGradient}

In this section we present an elliptic inequality involving the functions $W$ and $\bar W$, defined in \eqref{definW}, used in the proof of Theorem \ref{theo:Gradient}. In particular, we prove the following result.

\begin{lemma}\label{lemma_ineq}
	\begin{equation}\label{eq:Fbeta}
			F_\beta = \left(\frac{s_k (\Psi)}{\bar U ' (\Psi)}\right)^{2 \frac{n-1}{n}} (W-\bar W) \quad \textup{in } \quad \U .
	\end{equation}
	Then $F_\beta$ satisfies the following elliptic inequality in $\U$:
	\begin{equation}\label{ineqF}
		\Delta F_\beta - 2 \frac{n-1}{n} \lambda (\Psi) \pscalar{\nabla F_\beta}{\nabla u}- 2 \frac{n-1}{n}\frac{\abs{\nabla u}^2}{\bar U ' (\Psi)^2} \mu (\Psi) F_\beta \geq 0,
	\end{equation}
	where we define
	\begin{equation}\label{eq:deflambda}
		\lambda (r) := \frac{-\bar{U} '' (r)+ \textup{cot}_k (r) \bar U' (r)}{\bar{U}' (r)^2}, \quad \forall r \in [\bar r_-, \bar R ) \cup (\bar R, \bar r_+],
	\end{equation}
	and
	\begin{equation}\label{eq:defmu}
\mu (r):= f'(\bar U (r))-n k + \frac{n+2}{n} \lambda(r) f(\bar U (r)).
\end{equation}
	Then $\mu(r)$ is non-negative in $(\bar r_-, \bar R) \cup (\bar R, \bar r_+)$.
\end{lemma}

\begin{proof}
From now on, following the notation of \cite{EMa}, we denote $\chi_{\pm}=\chi$ and carry out the computations in a unified way for both monotonicity branches. We also denote derivatives with respect to $u$ by a dot and derivatives with respect to $\Psi$ by $'$.

The proof follows as in \cite[Theorem 3.5]{ABM}. The idea is to derive an elliptic inequality involving the functions $W$ and $\bar W$ defined in Section \ref{sectionGradient}. We recall that
\[
W = \abs{\nabla u}^2 \quad \textup{and} \quad \bar W = \abs{\nabla \bar u}^2 \circ \Psi \quad \textup{in } \U,
\]
where $\Psi : \U \to (\bar r_-, \bar R) \cup (\bar R, \bar r_+)$ is the pseudo-radial function defined in Definition \ref{def:pseudo-radial}.
We also recall that $\bar U$ satisfies
\begin{equation}
\bar U''(r) + (n-1)\cot_k(r)\bar U'(r)+f(\bar U)=0.
\end{equation}
From the definition of $\Psi$, we have $u = \bar U \circ \Psi$ on $\U$. Recall that $F_\beta = \beta \cdot (W- \bar W)$, where
\[
\beta (p) = \left(\frac{s_k (\Psi (p))}{\bar U' (\Psi (p))}\right)^{2\frac{n-1}{n}}, \quad \forall p \in \U.
\]
Using that $\textup{Ric} \geq (n-1)k g$, it follows from the Bochner formula that
\[
\frac{1}{2}\Delta \abs{\nabla u}^2 \geq \abs{\nabla^2 u}^2 + \left( (n-1)k-f'(u) \right) \abs{\nabla u}^2.
\]
Then, by expanding the squared norm
\begin{equation}\label{eq:norm}
	\abs{ \nabla^2 u + \lambda (\Psi) du \otimes du + \frac{1}{n} \left(f(u)-\lambda (\Psi) \abs{\nabla u}^2\right) g }^2 \geq 0,
\end{equation}
we obtain, after simplification, the inequality
\[
\Delta F_\beta - 2 \frac{n-1}{n} \lambda (\Psi) \pscalar{\nabla F_\beta}{\nabla u}- 2 \frac{n-1}{n}\frac{\abs{\nabla u}^2}{\bar U ' (\Psi)^2} \mu (\Psi) F_\beta \geq 0
\]
in $\U$. Moreover, equality at some point forces equality in \eqref{eq:norm}. Thus, to ensure the inequality has the correct sign structure for the maximum principle, it remains to show that $\mu \geq 0$ in $(\bar r_-, \bar R) \cup (\bar R, \bar r_+)$ under hypothesis \eqref{conditionf}.

First, note that up to scaling the solution and the metric, we may restrict to the cases $k=-1,0,$ or $1$. Observe that when $k =-1,0$ and $r \in (\bar r_-, \bar R)$, or when $k=1$ and $r \in (\bar r_-, \bar R)$, or when $k=1$ and $r \in [\bar r_1/2, \bar r_+)$ provided $\bar r_+ > \bar r_1/2$, one can check directly that $\mu (r) >0$ using \eqref{ODERadial}. For the remaining cases, consider the change of variable
\[
r(t) = \left\{\begin{matrix}
	\textup{arccosh} (t) & \text{ if } & k=-1\\[1mm]
	\sqrt{2 t} & \text{ if } & k=0 ,\\[1mm]
	\textup{arccos} (t) & \text{ if } &  k=1, 
\end{matrix}\right. 
\]
and define $\bar V (t) := \bar U (r (t))$. A direct computation shows that 
\[
\mu (t) =\frac{1}{r' (t)^2} \left( \frac{n+2}{n} \bar V'' (t)^2 -\bar V''' (t) \bar V ' (t) \right), \quad \forall t \in (\bar R, \bar r_+).
\]
The idea is to argue as in item (6) in \cite[Theorem 4.1]{EMa2}. Since the three cases $k=-1,0,1$ are very similar, we only write the proof for $k=0$. 

Define $\bar T = \bar R^2 /2$ and $\bar t_+ = \bar r_+^2 /2$. Using that $\bar U$ solves \eqref{ODERadial}, we obtain that $\bar V$ solves
\[
2 t \bar V'' (t) + n \bar V' (t) + f(\bar V) =0,
\]
and thus $\bar V'$ solves
\[
2 t \bar V'''(t) + (n+2) \bar V'' (t)+ f' (\bar V) \bar V' (t) =0.
\]
Integrating this equation between $\bar T$ and $t$, we obtain
\[
\bar V'' (t) = - \frac{1}{(2 t)^{\frac{n+2}{2}}} \left( (2 \bar T)^{\frac{n}{2}} f(M) + \int_{\bar T}^t (2 s)^{\frac{n}{2}} f' (\bar V (s)) \bar V' (s) \, ds \right),
\]
where $M = u_{\textup{max}}$. Note that $f' \geq 0$ and $\bar V' (t) < 0$ for $t \in (\bar T, \bar t_+)$, so the integral term is non-positive on this interval. Now take $\bar T < a < b$. Then
\[
\begin{split}
	\bar V '' (a)- \bar V'' (b) &= \left(\frac{1}{(2b)^{\frac{n+2}{2}}}-\frac{1}{(2a)^{\frac{n+2}{2}}}\right) \left( (2\bar T)^{\frac{n}{2}} f(M) - \int_{\bar T}^{a} (2s)^{\frac{n}{2}} f' (\bar V (s)) \bar V' (s) \, ds \right) \\
	 & \,\,\,\,\,\, +\frac{1}{(2b)^{\frac{n+2}{2}}} \int_{a}^{b} (2s)^{\frac{n}{2}} f' (\bar V (s)) \bar V' (s) \, ds \leq 0.
\end{split}
\]
This implies that $\bar V''' \geq 0$ on $(\bar T, \bar t_+)$, hence $\mu \geq 0$ on this interval. This concludes the proof.
\end{proof}

\section*{Data availability}
The code fowr the graphs are available in the following GitHub repository \url{https://github.com/FernanGI/ComparisonTechniquesLowerBound}. 
\section*{Acknowledgments}

J.M. Espinar is partially supported by the {\it Maria de Maeztu} Excellence Unit IMAG, reference CEX2020-001105-M, funded by MCINN/AEI/10.13039/ 501100011033/CEX2020-001105-M, and Spanish MIC Grant PID2024-160586NB-I00.

F. González-Ibáñez is partially supported by the  MCIU/AEI/10.13039/501100011033, the FSE+ and Spanish MIC Grant PID2023-150727NB-I00.
 
 D.A. Marín is partially supported by the {\it Maria de Maeztu} Excellence Unit IMAG, reference CEX2020-001105-M, funded by MCINN/AEI/10.13039/ 501100011033/CEX2020-001105-M, and Spanish MIC Grant PID2023-150727 NB.I00.

\bibliographystyle{amsalpha}
\bibliography{references}
\end{document}